\documentclass{article}

\usepackage{authblk}

\usepackage{graphicx}
\usepackage{amsmath}
\usepackage{amssymb}
\usepackage{amsthm}
\usepackage{tikz}
\usepackage{amsrefs}
\usetikzlibrary{arrows,shapes,graphs,fit,patterns}

\newtheorem{theorem}{Theorem}[section]
\newtheorem{lemma}[theorem]{Lemma}

\theoremstyle{definition}

\newtheorem{example}[theorem]{Example}

\theoremstyle{remark}

\numberwithin{equation}{section}

\DeclareMathOperator{\conv}{conv}

\DeclareMathOperator{\vertop}{vert}
\DeclareMathOperator{\aff}{aff}
\newcommand{\Rd}{\mathbb{R}^d}
\newcommand{\Pzd}{\mathcal{P}^d_0}
\newtheorem{corollary}[theorem]{Corollary}
\newtheorem{question}[theorem]{Question}

\begin{document}

\title{Self-polar polytopes}

\author{Alathea Jensen}


\affil{Department of Mathematics and Computer Science\\
Susquehanna University}

\date{April 17, 2018}

%

\maketitle

\begin{abstract}
Self-polar polytopes are convex polytopes that are equal to an orthogonal transformation of their polar sets.  These polytopes were first studied by Lov\'{a}sz as a means of establishing the chromatic number of distance graphs on spheres, and they can also be used to construct triangle-free graphs with arbitrarily high chromatic number.  We investigate the existence, construction, facial structure, and practical applications of self-polar polytopes, as well as the place of these polytopes within the broader set of self-dual polytopes.
\end{abstract}

\section*{Acknowledgments}

This is a pre-print of an article accepted for publication in the Contemporary Mathematics series of the AMS.

\section[Introduction]{Introduction}

Convex polytopes are fundamental objects in the field of discrete geometry that have been studied since ancient times.  They arise naturally as the feasible sets of systems of linear inequalities, and are also valued for their ability to encode complex combinatorial information of various sorts.

The faces of a polytope are its vertices, edges, facets, and so on, and together they can be arranged in a lattice, partially ordered by inclusion, which is known as the face lattice.  The face lattice is also the combinatorial type of the polytope, so that many different polytopes which are not equal as sets in real space may nevertheless have the same combinatorial type.

A great deal has been written about the ways in which a combinatorial type can be realized, including characterizations of realization spaces \cite{mnev}, discussions of whether certain combinatorial types can be realized with rational coordinates \cite{sturmfels}, and procedures for determining whether a given lattice is the face lattice of a polytope \cite{bokowski}.

All polytopes have a dual polytope whose face lattice is the dual of the original's face lattice, and some polytopes are also self-dual.  Much study has also been given to self-dual polytopes, including the enumeration of types in low dimensions \cite{dillencourt}, the discovery that self-duality is not necessarily involutory \cite{ashley}, and the classification of self-dualities into internal and external types \cite{cunmixer}.

No one has yet devoted any major study, however, to the topic of realizations of self-dual polytopes.  In this article, we examine this topic from the point of view of self-polar polytopes.  Self-polar is a term we have coined to describe any set that is an orthogonal transformation of its own polar set.  Thus, self-polar polytopes are a subset of self-dual polytopes.  They include as a subset the strongly self-dual polytopes of Lov\'{a}sz \cite{lovasz}, whose work inspired this study.

This article will investigate the basic properties of self-polar polytopes: their existence, construction, facial structure, symmetries, and applications.  We will focus in particular on polytopes that are equal to the negative of their polar sets.  Our ultimate question is whether all self-dual polytopes are self-polar.

We begin with the necessary definitions and preliminary information in Section \ref{sec:defs}, then discuss some properties of the orthogonal transformations in Section \ref{sec:selfpolarity}.  In Section \ref{sec:lowdims}, we discuss self-polar polytopes in two and three dimensions, and in Section \ref{sec:higherandlower} we describe ways to build self-polar polytopes in higher dimensions from ones in lower dimensions, and vice versa.  In Section \ref{sec:vertexnumbers}, we characterize the vertex numbers of negatively self-polar polytopes.  In Section \ref{sec:modifications}, we describe a way to construct a self-polar polytope from a polytope that is contained in an orthogonal transformation of its polar set, as well as a way to add vertices to a self-polar polytope while maintaining self-polarity.  In Section \ref{sec:applications}, we discuss applications of self-polar polytopes to graph coloring and to indicator function algebras. Finally, in Section \ref{sec:final}, we summarize our findings and propose future work.

\section{Definitions and Preliminaries}\label{sec:defs}

\subsection{Polytopes}
A \emph{polytope}  is any subset of real space $P\subset\Rd$ which can be described as the convex hull of a finite set of points in $\Rd$.  Although many different sets of points may yield the same convex hull, we can always find a set $\{v_1,v_2,\dots,v_n\}\in\Rd$ which is minimal in the sense that leaving out any $v_i$ would not yield $P$.  These points are called the \emph{vertices} of the polytope and are denoted $\mathrm{vert}(P)$.

By the \emph{dimension} of a polytope, denoted $\mathrm{dim}(P)$, we mean the dimension of the affine span of the polytope.  We will not necessarily assume that a polytope $P\subset\Rd$ is $d$-dimensional, but if it is, then we will say that $P$ is \emph{full-dimensional}.

In addition to being described by their vertices, all polytopes can also be described as the intersection of a finite number of closed halfspaces $\{H_1,H_2,\dots,H_m\}$.  A closed \emph{halfspace} $H$ is the closure of the solution set to a linear inequality in $\Rd$; that is, $H=\{x\in\Rd : \langle a,x\rangle\leq b\}$ for some $a\in\Rd$ and $b\in\mathbb{R}$.

Likewise, any finite intersection of halfspaces which is bounded is a polytope.  By bounded, we mean not containing any sequence of points that tend to infinity.  A finite intersection of halfspaces which is not bounded is called \emph{polyhedral}, but it is not considered a polytope.

The \emph{boundary of a halfspace} $H$, which we will denote $\partial H$, is the set of points in $H$ that satisfy the linear inequality with equality; that is, $\partial H=\{x\in\Rd : \langle x,a\rangle = b\}$.

\subsection{Faces}

Suppose for a polytope $P\in\Rd$ that we have a minimal list of its  closed halfspaces, $\{H_1,H_2,\dots,H_m\}$.  By minimal, we mean that to leave any $H_i$ out of the intersection would not yield $P$.  The \emph{boundary of the polytope}, which we will denote $\partial P$, is the set of points in $P$ which belong to one or more of the halfspace boundaries.  In other words, $\partial P=P\cap(\partial H_1\cup\partial H_2\dots\cup\partial H_m)$.

A \emph{face} $f$ of the polytope $P$ is any subset $f\subseteq P$ which can be written as $f=P\cap\partial H_{i_1}\cap\partial H_{i_2}\cap\dots\cap\partial H_{i_j}$ for some subset $\{H_{i_1},H_{i_2},\dots,H_{i_j}\}$ of the halfspaces.  If the subset of halfspaces is empty, then we get $P$ itself.  If we use only one halfspace, $P\cap\partial H_i$, then we get a face of one less dimension than $P$, called a \emph{facet}. 

Faces of two dimensions less than $P$ are called \emph{ridges}, faces of one dimension are called \emph{edges}, and faces of zero dimensions are the  vertices.  The empty set is also a face, and is considered to have $-1$ dimension.

\subsection{The Face Lattice}

The faces of $P$ form a partially ordered set under the subset relation.  In fact, the poset is a graded lattice, called the \emph{face lattice}, which we will denote $\mathcal{F}(P)$.  This lattice has $P$ as the maximum element, $\emptyset$ as the minimal element, and all elements graded by their dimensions.  Just like $P$ itself, each element of the lattice can be described either by the set of halfspaces whose boundaries it is contained in, or by the set of vertices that it contains.

A polytope's face lattice is its \emph{combinatorial type}, so that any two polytopes whose face lattices are isomorphic have the same combinatorial type.  Conversely, a polytope $P\subset\Rd$ with a particular combinatorial type is said to be a \emph{realization} of that combinatorial type.

The \emph{$f$-vector} of a polytope $P$, denoted $f(P)=(f_0,\dots,f_{d-1})$, records how many faces there are of each dimension.  Each component $f_i$ is the number of $i$-dimensional faces of $P$.  The Euler-Poincar\'{e} formula (see \cite{ziegler}) tells us that for any polytope $P$, $f(P)$ satisfies
\[-f_{-1}+f_0-f_1+f_2-\dots+(-1)^d f_d=0\]

Here, $f_{-1}$ and $f_d$ are defined analogously to the elements of the $f$-vector as the number of $(-1)$-dimensional faces and the number of $d$-dimensional faces respectively.

\subsection{Dual Polytopes}
The \emph{dual of a lattice} $\mathcal{L}$ is another lattice $\mathcal{L}^*$ which is the same except that the relation has been reversed.  In other words, there exists a bijection $\phi :\mathcal{L}\to\mathcal{L}^*$ which is ordering-reversing, so that for all $f,g\in\mathcal{L}$, $f\subset g$ if and only if $\phi(g)\subset\phi(f)$.  Such a bijection is known as a \emph{dual isomorphism}.

A \emph{self-dual lattice} is a lattice that is its own dual.  In the case of self-duality, the order-reversing bijection $\phi$ is a \emph{dual automorphism}.  The \emph{rank} of a dual automorphism is its order or period when regarded as a map; in other words, the least integer $r$ such that $\phi^r$ is the identity.

Two polytopes are said to be \emph{dual} to one another if and only if their face lattices are dual to one another.  Note that while realizations of polytopes can be dual to one another, the \emph{dual of a polytope} refers to a combinatorial type, rather than to any particular realization.  A \emph{self-dual polytope} is a polytope whose face lattice is self-dual.

It is far from obvious that the dual of every polytope's face lattice can be realized as a polytope, but this is indeed the case (see \cite{grunbaum} and \cite{ziegler}).

\subsection{The Polar Operation}

The \emph{polar} of a set $A\subseteq\Rd$ is $A^\circ=\{x\in\Rd:\langle x,a\rangle\leq 1 \textup{ for all } a\in A\}$.  As the polar operation forms the basis for the majority of the research in this article, we will list some of its properties in Lemma \ref{polarproperties}.

Before listing the polar's properties, however, we need to define one last piece of notation that will be of much utility throughout this article.

For a set $A\subseteq\Rd$, we will use $[A]$ to denote the closure of the convex hull of $A$ with the origin; that is, $[A]=\mathrm{closure}(\conv(A\cup\{0\}))$.  By closure, we mean the inclusion of all limit points under the Euclidean metric.   We will refer to this operation as the \emph{polar closure}, and we will say that $A$ is \emph{closed with respect to the polar}.  The polar closure operation is indeed a closure operation because Euclidean closure, convex hull, and union with zero are all themselves closure operations.

\begin{lemma}\label{polarproperties}
For any $A,B\subseteq\Rd$,
\begin{enumerate}
\item $A^\circ=[A^\circ]$.
\item $A^\circ=[A]^\circ$.
\item $A^{\circ\circ}=[A]$.
\item $A^{\circ\circ\circ}=A^\circ$.
\item $A\subset B\implies B^\circ\subset A^\circ$
\item $(A\cup B)^\circ=A^\circ\cap B^\circ$
\item $(A\cap B)^\circ=[A^\circ\cup B^\circ]$
\end{enumerate}
\end{lemma}
Standard references for these and other properties are \cite{grunbaum} and \cite{ziegler}.

\begin{lemma}\label{polarofmatrixlemma}
For any $A\subseteq\Rd$ and an invertible matrix $M\in\mathbb{R}^{d\times d}$, $(MA)^\circ=M^{-T}A^\circ$.  In particular, if $M$ is orthogonal, then $(MA)^\circ=MA^\circ$.
\end{lemma}
\begin{proof}
\begin{align*}
(MA)^\circ &= \{x\in\Rd : \langle x,Ma\rangle\leq 1 \textup{ for all } a\in A\}\\
&= \{x\in\Rd : \langle M^Tx,a\rangle\leq 1 \textup{ for all } a\in A\}
\end{align*}
Then, if we let $y=M^Tx$, so that $x=M^{-T}y$, this becomes
\begin{align*}
(MA)^\circ &= \{M^{-T}y\in\Rd : \langle y,a\rangle\leq 1 \textup{ for all } a\in A\}\\
&= M^{-T}\{y\in\Rd : \langle y,a\rangle\leq 1 \textup{ for all } a\in A\}\\
&= M^{-T}A^\circ
\end{align*}
As for the second statement, an orthogonal matrix $M$ is defined by the property that $M^{-T}=M$.
\end{proof}

\subsection{Polar Polytopes}

Now we return to our discussion of polytopes and note a few further properties of the polar as it applies to polytopes.

Let $\Pzd$ denote the set of all $d$-dimensional polytopes in $\Rd$ that contain the origin in their interior.  Note that $P=[P]$ for all $P\in\Pzd$, which is precisely why this set of polytopes is of interest.

It should be clear that every polytope can be realized as an element of $\Pzd$ for some value of $d$.  To obtain such a realization for a given polytope, we simply restrict the ambient space to the polytope's affine span to make it full-dimensional, and then translate the polytope so that the origin is in its interior.

\begin{lemma}
The polar operation is an involution on $\Pzd$; that is, if $P\in\Pzd$, then $P^\circ\in\Pzd$ and $P^{\circ\circ}=P$.
\end{lemma}
\begin{proof}
For $P\in\Pzd$, $P=[P]$, so $P=P^{\circ\circ}$ is immediate from Lemma \ref{polarproperties}(3).  From Lemma \ref{polarproperties}(2), $(\vertop(P))^\circ=[\vertop(P)]^\circ=P^\circ$.  We know $(\vertop(P))^\circ$ is a finite intersection of halfspaces, so $P^\circ$ is a polytope if it is bounded.  Since the origin is in the interior of $P$, $P$ contains an origin-centered ball of radius $\varepsilon>0$.  Then for all $y\in P^\circ$, $\frac{y}{|y|}\varepsilon\in P$.  Hence $\langle y,\frac{y}{|y|}\varepsilon\rangle\leq 1$.  Since $\langle y,\frac{y}{|y|}\varepsilon\rangle=|y|\varepsilon$, we have that $|y|\leq \frac{1}{\varepsilon}$, so $P^\circ$ is bounded.
\end{proof}

\begin{lemma}\label{duallemma}
For any polytope $P\in\Pzd$, $P^\circ$ and $P$ are dual polytopes.
\end{lemma}
For the proof of Lemma \ref{duallemma}, we refer the reader to \cite{grunbaum} and \cite{ziegler}.

\section{Self-Polarity}\label{sec:selfpolarity}

Since polar polytopes are realizations of dual polytopes, and there are many polytopes that are self-dual, it is natural to wonder whether there are any polytopes that are self-polar.  It turns out that the answer to this question depends on what we mean by ``self-polar''.  If we mean $P=P^\circ$, then the answer is the following.

\begin{theorem}\label{unitballtheorem}
The only set $A\subseteq\Rd$ for which $A=A^\circ$ is the unit ball, $A=\{x\in\Rd:|x|\leq 1\}$.
\end{theorem}
\begin{proof}
Let $B$ denote the unit ball; that is, $B=\{x\in\Rd:|x|\leq 1\}$.  It should be clear that $B=B^\circ$ from the definition of the polar operation.  Now suppose we have some other $A\subseteq\Rd$ such that $A=A^\circ$.  For all $x\in A$, we also have $x\in A^\circ$, so we must have $\langle x,x\rangle\leq 1\implies|x|\leq 1$.  Hence $A\subseteq B=B^\circ$.  Then using Lemma \ref{polarproperties}(5) on $A\subseteq B^\circ$, we get $B\subseteq A^\circ=A$.  Hence $A=B$.
\end{proof}

In light of this, we must use a slightly more relaxed definition of self-polarity in order to find anything interesting to work with.

We define a set $A\subset\Rd$ as \emph{self-polar} provided there exists some orthogonal transformation $U$ of $\Rd$ such that $A=UA^\circ$.  By orthogonal transformation, we mean any nonsingular linear transformation $U$ such that $U^{T}=U^{-1}$.  In general, these are rotations and reflections.  In the case that $A=-A^\circ$, we say that $A$ is \emph{negatively self-polar}.

We will refer to such an orthogonal transformation as a \emph{self-polarity map} of the set $A$.  In the case of a polytope $P$, note that every self-polarity map induces a dual automorphism of the face lattice via the mapping from each face of $P$ to its dual face in the polar polytope $P^\circ$, to the image of that dual face in $UP^\circ$ after the orthogonal transformation, which is again some face of $P$.  Hence a given polytope can have at most as many distinct self-polarity maps as its face lattice has distinct dual automorphisms.

We will soon begin constructing many examples of self-polar polytopes, so it is worthwhile to have a simple way of verifying the property of self-polarity.  The next theorem fulfills that goal.

\begin{theorem}\label{selfpolarconditions}
For a polytope $P\in\Pzd$ and an orthogonal transformation $U$, $P=UP^\circ$ if and only if both of the following are true.
\begin{enumerate}
\item For all vertices $v,w$ of $P$, $\langle Uv,w\rangle\leq 1$.
\item For each facet $F$ of $P$, there is a vertex $v$ of $P$ such that $\langle Uv,w\rangle=1$ for all vertices $w\in\vertop(F)$.
\end{enumerate}
\end{theorem}
\begin{proof}
The first condition guarantees that $P\subseteq UP^\circ$.  As for the second condition, each vertex $v$ of $UP^\circ$ is generated by a facet $F$ of $P$, so that $\langle v,Uw\rangle=1$ for all vertices $w\in\vertop(F)$.  Thus the second condition guarantees that we must have $\vertop(UP^\circ)\subseteq P$.  Hence $P=UP^\circ$.
\end{proof}

Before proceeding to show examples of self-polar polytopes and explore their properties, it is interesting to note that we can say something about the properties of the self-polarity maps of polytopes even without knowing anything about the polytopes themselves.  First, recall that a matrix $M$ is \emph{periodic} if there is some $n\in\mathbb{N}$ such that $M^n$ is the identity.  The \emph{period of a matrix} $M$ is the minimum $n\in\mathbb{N}$ such that $M^n$ is the identity.

\begin{theorem}\label{Uconditionstheorem}
If $P=UP^\circ$ for some polytope $P\in\Pzd$ and some orthogonal map $U$, then $U^2P=P$ and $U$ is periodic with an even period.
\end{theorem}
\begin{proof}
Assume the conditions of the theorem statement and then apply the polar operation to both sides of $P=UP^\circ$ to obtain $P^\circ=(UP^\circ)^\circ$.  Using Lemmas \ref{polarofmatrixlemma} and \ref{polarproperties}(3), this becomes $P^\circ=UP^{\circ\circ}=UP$.  Then, substituting $UP$ for $P^\circ$ in $P=UP^\circ$, we get $P=U^2P$.  Hence $U^2$ is a symmetry of $P$.

Any symmetry of a polytope must have a finite period.  This is so because a symmetry maps vertices to vertices.  Each $U^2,U^4,U^6,\dots$ induces a permutation on the vertices of $P$, and there are a finite number of such permutations.  Hence there exist some $i,j\in\mathbb{N}$ with $i<j$ such that $U^{2i}$ and $U^{2j}$ induce the same permutation of the vertices of $P$.  But since the vertices of $P$, regarded as vectors, span the space $\Rd$, we must have $U^{2i}=U^{2j}$, which implies $I=U^{2(j-i)}$.  Thus $U$ is periodic.

Now suppose that $U$ has an odd period, so that there is some $k\in\mathbb{N}$ such that $U^{2k-1}=I$.  Then $U^{2k-1}P=P$.  But we also know that $U^{2k}P=P$, so together these imply $U^{2k-1}P=U^{2k}P$, which implies $P=UP$.  Then substituting into $P^\circ=UP$ we get $P^\circ=P$, which by Theorem \ref{unitballtheorem} implies that $P$ is a unit ball, contradicting our assumption that $P$ is a polytope.  Hence $U$ must have an even period.
\end{proof}

\section{In Low Dimensions}\label{sec:lowdims}
\subsection{Two Dimensions}

In two dimensions, it is well-known that every polygon is self-dual.  If we start with a regular polygon, centered at the origin, we obtain as its polar a dilated and rotated copy of the same polygon.  Let us refer to the distance between origin and vertex in such a polygon as its \textit{radius}.  When we increase the radius of such a polygon, the radius of its polar decreases, and vice versa.  See Figure \ref{squaredilations} for examples.  Hence there is always some radius for which the polar radius is the same as the original radius.  

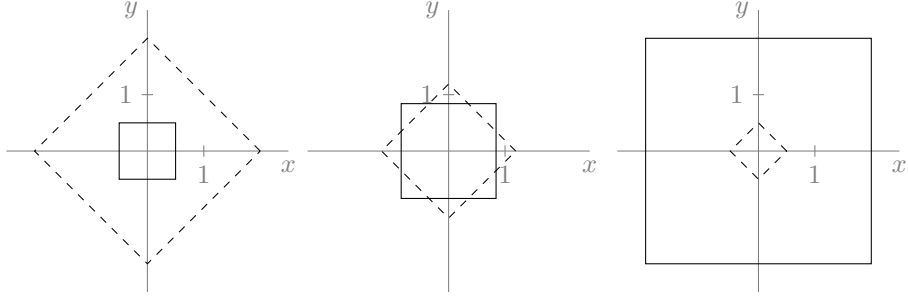
\begin{figure}

\begin{minipage}[c]{0.33\textwidth}
\centering
\begin{tikzpicture}[>=angle 60,x=0.75cm,y=0.75cm]
\draw[gray] (-2.5,0) -- (2.5,0) node[below] {$x$};
\draw[gray] (0,-2.5) -- (0,2.5) node[left] {$y$};
\draw[gray] (1,0.1) -- (1,-0.1) node[below] {1};
\draw[gray] (0.1,1) -- (-0.1,1) node[left] {1};
\draw (-0.5,-0.5) -- (-0.5,0.5) -- (0.5,0.5) -- (0.5,-0.5) -- cycle;
\draw[dashed] (2,0) -- (0,2) -- (-2,0) -- (0,-2) -- cycle;
\end{tikzpicture}
\end{minipage}
\begin{minipage}[c]{0.33\textwidth}
\centering
\begin{tikzpicture}[>=angle 60,x=0.75cm,y=0.75cm]
\draw[gray] (-2.5,0) -- (2.5,0) node[below] {$x$};
\draw[gray] (0,-2.5) -- (0,2.5) node[left] {$y$};
\draw[gray] (1,0.1) -- (1,-0.1) node[below] {1};
\draw[gray] (0.1,1) -- (-0.1,1) node[left] {1};
\draw (-0.8409,-0.8409) -- (-0.8409,0.8409) -- (0.8409,0.8409) -- (0.8409,-0.8409) -- cycle;
\draw[dashed] (1.1892,0) -- (0,1.1892) -- (-1.1892,0) -- (0,-1.1892) -- cycle;
\end{tikzpicture}
\end{minipage}
\begin{minipage}[c]{0.33\textwidth}
\centering
\begin{tikzpicture}[>=angle 60,x=0.75cm,y=0.75cm]
\draw[gray] (-2.5,0) -- (2.5,0) node[below] {$x$};
\draw[gray] (0,-2.5) -- (0,2.5) node[left] {$y$};
\draw[gray] (1,0.1) -- (1,-0.1) node[below] {1};
\draw[gray] (0.1,1) -- (-0.1,1) node[left] {1};
\draw (-2,-2) -- (-2,2) -- (2,2) -- (2,-2) -- cycle;
\draw[dashed] (0.5,0) -- (0,0.5) -- (-0.5,0) -- (0,-0.5) -- cycle;
\end{tikzpicture}
\end{minipage}
\caption{Examples of a regular, origin-centered square (solid) and its polar set (dashed) with three different radii.}
\label{squaredilations}\end{figure}

When their radii are the same, a regular, origin-centered polygon and its polar can be made identical by a suitable rotation.  We can also make the two polygons identical by a reflection, if we let the axis of symmetry connect the origin and any one of the points of intersection between the boundary of the polygon and the boundary of its polar.  See Figure \ref{squarepolarities} for an example.

\begin{figure}
\begin{minipage}[c]{0.5\textwidth}
\centering
\begin{tikzpicture}[>=angle 60,x=1.5cm,y=1.5cm]
\draw[gray] (-1.5,0) -- (1.5,0) node[below] {$x$};
\draw[gray] (0,-1.5) -- (0,1.5) node[left] {$y$};
\draw[gray] (1,0.1) -- (1,-0.1) node[below right] {1};
\draw[gray] (0.1,1) -- (-0.1,1) node[above left] {1};
\draw (-0.8409,-0.8409) -- (-0.8409,0.8409) -- (0.8409,0.8409) -- (0.8409,-0.8409) -- cycle;
\draw[dashed] (1.1892,0) -- (0,1.1892) -- (-1.1892,0) -- (0,-1.1892) -- cycle;
\draw[<->,red] (0.8,0.9) arc (50:90:1.1);
\end{tikzpicture}
\end{minipage}
\begin{minipage}[c]{0.5\textwidth}
\centering
\begin{tikzpicture}[>=angle 60,x=1.5cm,y=1.5cm]
\draw[gray] (-1.5,0) -- (1.5,0) node[below] {$x$};
\draw[gray] (0,-1.5) -- (0,1.5) node[left] {$y$};
\draw[gray] (1,0.1) -- (1,-0.1) node[below right] {1};
\draw[gray] (0.1,1) -- (-0.1,1) node[above left] {1};
\draw (-0.8409,-0.8409) -- (-0.8409,0.8409) -- (0.8409,0.8409) -- (0.8409,-0.8409) -- cycle;
\draw[dashed] (1.1892,0) -- (0,1.1892) -- (-1.1892,0) -- (0,-1.1892) -- cycle;
\draw[red] (0.5858,1.4142) -- (-0.5858,-1.4142);
\draw[<->,red] (0.5-0.2,1.2071+0.04142*2) -- (0.5+0.2,1.2071-0.04142*2);
\end{tikzpicture}
\end{minipage}
\caption{Examples of a self-polarity by a rotation of $45^\circ$ (left panel) and by reflection across the red line (right panel).}
\label{squarepolarities}\end{figure}
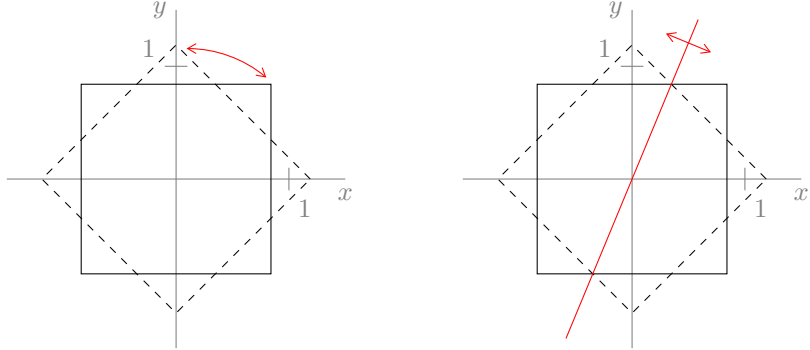

Since the dihedral group for an $n$-gon has $2n$ elements, there are evidently only $2n$ possible dual automorphisms of the $n$-gon's face lattice.  We will show that any $n$-gon can be realized in such a way that it is self-polar by $2n$ distinct orthogonal transformations, each of which realizes a different one of the possible dual automorphisms.

\begin{theorem} \label{polygonthm}
For $n\in\mathbb{N}$, an $n$-sided polygon can be realized as self-polar by $2n$ distinct orthogonal transformations, $n$ of which are reflections and the other $n$ of which are rotations.
\end{theorem}
\begin{proof}
Consider a regular polygon $P$ with $n$ sides, centered at the origin, and given by the vertices
$\{v_1,v_2,\dots,v_n\}$ where
\[v_i:=\left(\sqrt{\sec\left(\frac{\pi}{n}\right)} \cos\left(\frac{2\pi}{n}i\right), \sqrt{\sec\left(\frac{\pi}{n}\right)} \sin\left(\frac{2\pi}{n}i\right)\right)\]

Let $U$ be a counterclockwise rotation of $\pi/n$.  Note that $U$ is an orthogonal transformation.  Then each $Uv_i$ is given by
\[Uv_i:=\left(\sqrt{\sec\left(\frac{\pi}{n}\right)} \cos\left(\frac{2\pi}{n}i+\frac{\pi}{n}\right), \sqrt{\sec\left(\frac{\pi}{n}\right)} \sin\left(\frac{2\pi}{n}i+\frac{\pi}{n}\right)\right)\]

Consider that the maximum dot product between any $v_i$ and $Uv_j$ must be between vertices that are as close together as possible, hence pairs of the form $\langle v_i,Uv_i\rangle$ or $\langle v_i,Uv_{i-1}\rangle$.  Using trigonometric identities, we can find these dot products as
\begin{align*}
\langle v_i,Uv_i\rangle &= \sec\left(\frac{\pi}{n}\right)\left(\cos\left(\frac{2\pi}{n}i\right)\cos\left(\frac{2\pi}{n}i+\frac{\pi}{n}\right)+\sin\left(\frac{2\pi}{n}i\right)\sin\left(\frac{2\pi}{n}i+\frac{\pi}{n}\right)\right)\\
&= \sec\left(\frac{\pi}{n}\right)\cos\left(\frac{\pi}{n}\right)\\
&= 1
\end{align*}
and
\begin{align*}
\langle v_i,Uv_{i-1}\rangle &= \sec\left(\frac{\pi}{n}\right)\left(\cos\left(\frac{2\pi}{n}i\right)\cos\left(\frac{2\pi}{n}i-\frac{\pi}{n}\right)+\sin\left(\frac{2\pi}{n}i\right)\sin\left(\frac{2\pi}{n}i-\frac{\pi}{n}\right)\right)\\
&= \sec\left(\frac{\pi}{n}\right)\cos\left(\frac{\pi}{n}\right)\\
&= 1
\end{align*}

Hence the first condition of Theorem \ref{selfpolarconditions} is fulfilled.  As for the second condition, the facets of $P$ are of the form $\conv(v_{i-1},v_i)$.  Both these vertices have a dot product of 1 with $Uv_{i-1}$.  Hence $P=UP^\circ$.

Now consider that $U^2$ is a symmetry of $P$, because $P$ is regular with $n$ sides and $U$ is a rotation by $\pi/n$.  This means $U^{2j}P=P$ for all $j\in\mathbb{N}$, and so $P=UP^\circ\implies P=U(U^{2j}P)^\circ=U^{2j+1}P^\circ$.  Only $U,U^3,U^5,\dots,U^{2n-1}$ are distinct, however, because $U^{2n}=I$.  These are the $n$ distinct rotations mentioned in the theorem statement.

Let $R$ be reflection across the $x$-axis.  Then there are $n$ other symmetries of $P$, given by $R,U^2R,U^4R,U^6R,\dots,U^{2n-2}R$, and, just as before, $P=UP^\circ\implies P=U(U^{2j}RP)^\circ=U^{2j+1}RP^\circ$.  These are the $n$ distinct reflections mentioned in the theorem statement.
\end{proof}

Negatively self-polar polytopes are of particular interest due to their application to sphere-coloring.  From the last theorem, we can already say that odd polygons can be realized as negatively self-polar.

\begin{corollary}
Polygons with an odd number of sides can be realized as negatively self-polar.
\end{corollary}
\begin{proof}
For $\mathbb{R}^2$, the negative transformation $-I$ is equivalent to a rotation of $\pi$.  If $P$ is as described in the proof of Theorem \ref{polygonthm} and has $n$ sides, and $U$ is a rotation by $\pi/n$, then $U,U^3,U^5,\dots,U^{2n-1}$ are self-polarity maps of $P$.  If $n$ is odd, then $U^n$ must also be in this list.  Since $U^n$ is a rotation by $\pi$, we conclude that $P=-P^\circ$.
\end{proof}

Now let us consider even polygons.  First of all, if our goal is to construct a negatively self-polar even polygon, it ought to be clear that a regular even polygon, centered at the origin, is a hopeless case.  This is so because for any such realization, call it $P$, we have $P=-P$ and so $P^\circ=-P^\circ$, and so $P=-P^\circ$ would imply $P=P^\circ$, which is impossible except in the case of the unit ball.

It is not necessary to use a regular polygon to achieve self-polarity, however.  This opens the door to tremendous number of possible realizations, suggesting perhaps a linear algebra-based approach.  However, it is not necessary to turn to such techniques, because the reason for the impossibility of negatively self-polar even polygons is combinatorial, as we shall show.

Recall that any self-polarity map gives rise to a dual automorphism on the face lattice of the polytope.  For reflections, such a map is an involution.  We will show that for an even polygon, it is impossible to produce an involutory dual automorphism without mapping some facet to one of the vertices that it contains.

This is significant because when this is the case, it is impossible to realize the dual automorphism as the negative self-polarity map.  This impossibility is true of any polytope, not only the two-dimensional sort, so we will pause to state it as a lemma before applying it to the case of even polygons.

\begin{lemma} \label{dualautolemma}
A dual automorphism $\phi$ on the face lattice of a polytope cannot be realized as the negative self-polarity map when there exists a vertex $v$ for which $\phi(v)$ contains $v$.
\end{lemma}
\begin{proof}
In a negatively self-polar polytope, each vertex $v$ is paired with a facet $f$ which lies in the boundary of the half-space $\{x\in\Rd:\langle -v,x\rangle\leq 1\}$.  If $v$ is in fact an element of this facet, then $\langle -v,v\rangle=1$.  But $\langle -v,v\rangle=1$ implies $|v|^2=-1$, which is impossible.  So $v$ cannot be contained in $f$.  Hence any dual automorphism $\phi$ for which $\phi(v)$ contains $v$ cannot be realized by a negative self-polarity map.
\end{proof}

\begin{theorem}\label{evenpolygonthm}
No even polygon can be realized as negatively self-polar.
\end{theorem}
\begin{proof}
We begin by noting that an involutory dual automorphism partners each vertex with a facet, in such a way that if vertex $v$ is in facet $f$, then the partner vertex of $f$ must be contained in the partner facet of $v$.

Now suppose that we have an even polygon with our vertices labeled $0$, $1$, $2$, $\dots$, $n-1$, going in order around the polygon, and we have likewise labeled the facets by ordered pairs of the vertices they contain, so that facet $(i,i+1)$ contains vertices $i$ and $i+1$, and the addition operation in the expression $i+1$ is understood to be modulo $n$.

Now suppose we have a pairing of vertices and facets that corresponds to an involutory dual automorphism.  Suppose that vertex $0$ has facet $(i,i+1)$ as its partner.  Then vertex $i$ must be paired with a facet that contains $0$, either (Case 1) facet $(0,1)$ or (Case 2) facet $(n-1,0)$.

In Case 1, $0$ is paired with $(i,i+1)$ and $i$ with $(0,1)$.  Then $1$ must be paired with a facet that contains $i$, but there is only one of those that isn't paired yet: $(i-1,i)$. In fact, from this point on, every pairing is necessitated by the pairings that we have already looked at.  In the order that we can deduce them, and including those previously decided, they are:

\begin{align*}
0&\leftrightarrow(i,i+1)\\
i&\leftrightarrow(0,1)\\
1&\leftrightarrow(i-1,i)\\
i-1&\leftrightarrow(1,2)\\
2&\leftrightarrow(i-2,i-1)\\
i-2&\leftrightarrow(2,3)\\
\vdots
\end{align*}

See Figure \ref{evenpolygons} (left) for a diagram of the pairings.

\begin{figure}
\begin{minipage}[c]{0.5\textwidth}
\centering
\begin{tikzpicture}[>=angle 60,x=0.7cm,y=0.7cm]
\draw[fill] (135:3) circle[radius=0.1] node[above left] {$0$};
\draw[fill] (105:3) circle[radius=0.1] node[above] {$1$};
\draw[fill] (75:3) circle[radius=0.1] node[above] {$2$};
\draw[fill] (45:3) circle[radius=0.1] node[above right] {$3$};
\draw (150:3) -- (135:3) -- (105:3) -- (75:3) -- (45:3) -- (30:3);
\draw[fill] (-135:3) circle[radius=0.1] node[below left] {$i+1$};
\draw[fill] (-105:3) circle[radius=0.1] node[below left] {$i$};
\draw[fill] (-75:3) circle[radius=0.1] node[below right] {$i-1$};
\draw[fill] (-45:3) circle[radius=0.1] node[below right] {$i-2$};
\draw (-150:3) -- (-135:3) -- (-105:3) -- (-75:3) -- (-45:3) -- (-30:3);
\node at (0:3) {$\vdots$};
\node at (180:3) {$\vdots$};
\draw[<->,red] (135:2.9) -- (-120:2.9);
\draw[<->,red] (120:2.9) -- (-105:2.9);
\draw[<->,red] (105:2.9) -- (-90:2.9);
\draw[<->,red] (90:2.9) -- (-75:2.9);
\draw[<->,red] (75:2.9) -- (-60:2.9);
\draw[<->,red] (60:2.9) -- (-45:2.9);
\end{tikzpicture}
\end{minipage}
\begin{minipage}[c]{0.5\textwidth}
\centering
\begin{tikzpicture}[>=angle 60,x=0.7cm,y=0.7cm]
\draw[fill] (135:3) circle[radius=0.1] node[above left] {$n-3$};
\draw[fill] (105:3) circle[radius=0.1] node[above left] {$n-2$};
\draw[fill] (75:3) circle[radius=0.1] node[above right] {$n-1$};
\draw[fill] (45:3) circle[radius=0.1] node[above right] {$0$};
\draw (150:3) -- (135:3) -- (105:3) -- (75:3) -- (45:3) -- (30:3);
\draw[fill] (-135:3) circle[radius=0.1] node[below left] {$i+1$};
\draw[fill] (-105:3) circle[radius=0.1] node[below left] {$i$};
\draw[fill] (-75:3) circle[radius=0.1] node[below right] {$i-1$};
\draw[fill] (-45:3) circle[radius=0.1] node[below right] {$i-2$};
\draw (-150:3) -- (-135:3) -- (-105:3) -- (-75:3) -- (-45:3) -- (-30:3);
\node at (0:3) {$\vdots$};
\node at (180:3) {$\vdots$};
\draw[<->,red] (45:2.9) -- (-120:2.9);
\draw[<->,red] (60:2.9) -- (-105:2.9);
\draw[<->,red] (75:2.9) -- (-90:2.9);
\draw[<->,red] (90:2.9) -- (-75:2.9);
\draw[<->,red] (105:2.9) -- (-60:2.9);
\draw[<->,red] (120:2.9) -- (-45:2.9);
\end{tikzpicture}
\end{minipage}%
\caption{Diagram of pairings in Case 1 (left) and Case 2 (right) of Theorem \ref{evenpolygonthm}.}
\label{evenpolygons}\end{figure}
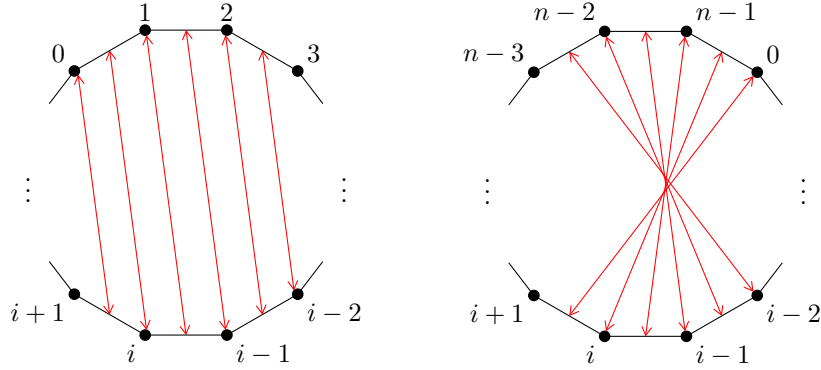

We see that there are two alternating patterns in which vertices $0,1,2,\dots$ are paired with facets $(i,i+1),(i-1,i),(i-2,i-1),\dots$ respectively, and vertices $i,i-1,i-2,\dots$ are paired with facets $(0,1),(1,2),(2,3),\dots$ respectively.  In both patterns, vertex $k$ is paired with facet $(i-k,i-k+1)$ (using addition modulo $n$), hence the patterns are consistent with one another.

Consider now the situation when $i$ is even and $k=\frac{i}{2}$.  Then vertex $k$ is paired with facet $(k,k+1)$.  On the other hand, if $i$ is odd, then when $k=\frac{i+1}{2}$, vertex $k$ is paired with facet $(k-1,k)$.  In either case, the pattern of pairings forces the existence of a vertex paired with one of the facets containing it, which, by Lemma \ref{dualautolemma}, means that this dual automorphism cannot be realized by a negative self-polarity.

Now for Case 2, in which $0$ is paired with $(i,i+1)$ and $i$ with $(n-1,0)$.  A similar situation arises here:
\begin{align*}
0&\leftrightarrow(i,i+1)\\
i&\leftrightarrow(n-1,0)\\
n-1&\leftrightarrow(i-1,i)\\
i-1&\leftrightarrow(n-2,n-1)\\
n-2&\leftrightarrow(i-2,i-1)\\
i-2&\leftrightarrow(n-3,n-2)\\
\vdots
\end{align*}

See Figure \ref{evenpolygons} (right) for a diagram of the pairings.

We see that there are two alternating patterns in which vertices $0,n-1,n-2,\dots$ are paired with facets $(i,i+1),(i-1,i),(i-2,i-1),\dots$ respectively, and vertices $i,i-1,i-2,\dots$ are paired with facets $(n-1,0),(n-2,n-1),(n-3,n-2),\dots$ respectively.  In the first pattern, vertex $k$ is paired with facet $(k+i,k+i+1)$, while in the second pattern, vertex $k$ is paired with facet $(k-i-1,k-i)$, where all of these expressions are modulo $n$.

These two patterns must be consistent with one another, of course.  The first pattern pairs vertex $0$ with facet $(i,i+1)$, while the second pairs vertex $0$ with facet $(n-i-1,n-1)$.  This must be the same facet, so we have $i=n-i-1$, which implies $n=2i+1$, meaning that $n$ is odd, contrary to our assumption.
\end{proof}

\subsection{Three Dimensions}

Having completed our investigation in two dimensions, we proceed to three dimensions.  Here, we begin with the happy discovery that, as with two dimensions, every self-dual three-dimensional polytope has a self-polar realization.

First we need a lemma.

\begin{lemma}\label{tangentlemma}
If the affine hull of a face $f$ of a polytope $P\in \Pzd$ is tangent to an origin-centered sphere of radius $r$ at point $x$, then the affine hull of the corresponding dual face $g$ of $P^\circ$ is tangent to an origin-centered sphere of radius $1/r$ at point $x/r^2$.
\end{lemma}
\begin{proof}
Let $x'$ be the point of tangency for $\aff(f)$.  Then $f$ must be orthogonal to $x'$, so every $x\in f$ can be expressed as $x=x'+w$ where $w$ is some vector orthogonal to $x'$.

Now consider $y'=x'/r^2$.  We know $y'$ must be in $\aff(g)$ because for any $x\in f$ there is some $w$ orthogonal to $x'$ so that we have $x=x'+w$, which gives us
\[\langle x, y'\rangle=\left\langle x'+w,\frac{x'}{r^2}\right\rangle=\frac{1}{r^2}\left(\langle x',x'\rangle+\langle w,x'\rangle\right)=\frac{1}{r^2}(r^2+0)=1\]

Now consider that any vector which lies in $g$ must be orthogonal to $y'$.  This is because any vector lying in $g$ is of the form $y_1-y_2$ for some $y_1,y_2\in g$, and we have
\begin{multline*}
\langle y', y_1-y_2\rangle=\left\langle\frac{x'}{r^2},y_1-y_2\right\rangle=\frac{1}{r^2}\langle x',y_1-y_2\rangle\\=\frac{1}{r^2}\left(\langle x',y_1\rangle-\langle x',y_2\rangle\right)=\frac{1}{r^2}(1-1)=0
\end{multline*}

Since $g$ is orthogonal to $y'$, and $|y'|=1/r$, the lemma's statement follows.
\end{proof}

\begin{theorem}
Every self-dual three-dimensional polytope has a self-polar realization.
\end{theorem}
\begin{proof}
The Koebe-Andreev-Thurston theorem implies that each 3-connected planar
graph can be realized by a 3-polytope which has all edges tangent to the unit sphere and that the edge-tangent realization for which the barycenter of the tangency points is the center of the sphere is unique up to orthogonal transformations.  See \cite{schramm} for more details and a proof.

For a given combinatorial type, let us call this realization $P$.  From Lemma \ref{tangentlemma}, the edges of $P^\circ$ will also be tangent to the unit sphere, with their tangency points in the same locations as those of $P$, hence having their barycenter at the origin.  If the polar $P^\circ$ is of the same combinatorial type as $P$, then its realization in this form is unique up to orthogonal transformation, hence we have $P=UP^\circ$ for some orthogonal transformation $U$ which maps edge-tangency points onto edge-tangency points.
\end{proof}

We would also like to know more specifically which three-dimensional polytopes can be realized as negatively self-polar.  It is not very difficult to construct a few examples that demonstrate we can have (almost) any number of vertices we like.

\begin{theorem} \label{d3types}
For $n=4$ and $n \geq 6$, there exists a 3-dimensional polytope $P=-P^\circ$ with $n$ vertices.
\end{theorem}

\begin{proof}
The proof is by construction of three different types of polytopes: those with an even number of vertices, those with vertex number congruent to $3\mod4$, and those with vertex number congruent to $1\mod4$.

For the first type, those with an even number of vertices, we construct a pyramid over an odd polygon.  For $k\in\mathbb{N}$, let $P$ be the convex hull of $(0,0,1)$ and
\[\left(\sqrt{2\sec\left(\frac{\pi}{2k+1}\right)}\cos\left(\frac{2\pi i}{2k+1}\right), \sqrt{2\sec\left(\frac{\pi}{2k+1}\right)}\sin\left(\frac{2\pi i}{2k+1}\right), -1\right)\]
for $i=1,2,3,...,2k+1$.  Then $P=-P^\circ$ and has $2k+2$ vertices.

For the second type, we construct a pyramid over an odd polygon, glued to a prism over the same polygon.  For $k\in\mathbb{N}$, let $P$ be the convex hull of $(0,0,1)$ and
\[\left(\sqrt{\sec\left(\frac{\pi}{2k+1}\right)}\cos\left(\frac{2\pi i}{2k+1}\right), \sqrt{\sec\left(\frac{\pi}{2k+1}\right)}\sin\left(\frac{2\pi i}{2k+1}\right), 0\right)\]
for $i=1,2,3,...,2k+1$, and
\[\left(\sqrt{\sec\left(\frac{\pi}{2k+1}\right)}\cos\left(\frac{2\pi i}{2k+1}\right), \sqrt{\sec\left(\frac{\pi}{2k+1}\right)}\sin\left(\frac{2\pi i}{2k+1}\right), -1\right)\]
for $i=1,2,3,...,2k+1$.  Then $P=-P^\circ$ and has $4k+3$ vertices.

For the third type, we construct a pyramid over an even polygon with deep truncations at each vertex of the base.  For $k\in\mathbb{N}$ and $k\geq2$, let $P$ be the convex hull of $\left(0,0,\cot\left(\frac{\pi}{2k}\right)\right)$ and
\[\left(\cos\left(\frac{\pi i}{k}\right), \sin\left(\frac{\pi i}{k}\right), 0\right)\]
for $i=1,2,3,...,2k$, and
\[\left(\sec\left(\frac{\pi}{2k}\right)\cos\left(\frac{\pi i}{k}+\frac{\pi}{2k}\right), \sec\left(\frac{\pi}{2k}\right)\sin\left(\frac{\pi i}{k}+\frac{\pi}{2k}\right), -\tan\left(\frac{\pi}{2k}\right)\right)\]
for $i=1,2,3,...,2k$.  Then $P=-P^\circ$ and has $4k+1$ vertices.
\end{proof}

Shown in Figure \ref{3dfigures} is an example with $k=2$ for each of the constructions.

\begin{figure}

\centering
\includegraphics[width=\textwidth]{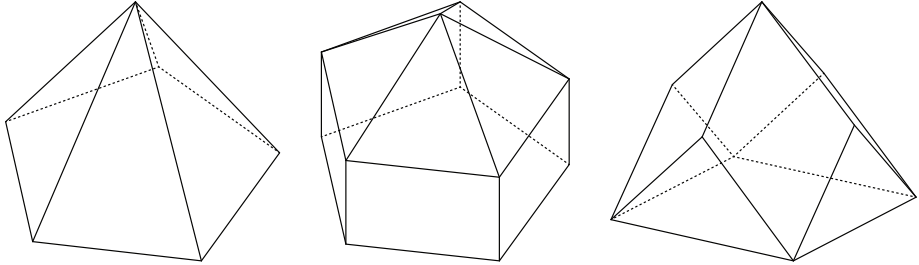}
\caption{Examples of the constructions in Theorem \ref{d3types} with $k=2$.}
\label{3dfigures}\end{figure}

As for five vertices, which was missing from Theorem \ref{d3types}, there are only two different combinatorial types of polytopes with 5 vertices in $\mathbb{R}^3$, and, of those, only the square pyramid is self-dual.  It turns out that the square pyramid cannot be realized as negatively self-polar because its base cannot be realized as negatively self-polar.  The proof of this fact appears a bit later, however, in Theorem \ref{negpyrbasethm}.

Now that we know how many vertices are possible for negatively self-polar three-dimensional polytopes, we also know the possible $f$-vectors $(f_0,f_1,f_2)$.

\begin{corollary}
If $P=-P^\circ$ for $P\in\mathcal{P}_0^3$, then $f(P)=(n,2n-2,n)$ for some $n\geq 4$ and $n\neq 5$.
\end{corollary}
\begin{proof}
The Euler-Poincar\'{e} formula tells us that $-1+f_0-f_1+f_2-1=0$, hence $f_1=f_0+f_2-2$.  We must have $f_0=f_2$ for any self-dual three-dimensional polytope, so this becomes $f_1=2f_0-2$.  We know the possible values of $n$ from Theorem \ref{d3types} and \ref{negpyrbasethm}.
\end{proof}

\section{In Higher and Lower Dimensions}\label{sec:higherandlower}

\subsection{Pyramids}

It was simple to construct concrete examples of self-polar polytopes in two dimensions, and not very difficult in three, but in higher dimensions, the path forward is less clear.  In order for us to climb from one dimension to the next, we begin with a pyramidal construction that allows us to construct self-polar pyramids in higher dimensions over self-polar bases.

\begin{theorem} \label{pyramidthm}
For a polytope $P\in\Pzd$ that is self-polar by orthogonal transformation $U$, a pyramid over $P$ can be realized as self-polar by the transformation $W$, where
\begin{equation*}
W:=
\begin{pmatrix}
U & 0 \\
0 & -1
\end{pmatrix}
\end{equation*}
\end{theorem}
\begin{proof}
First we will construct a realization of the pyramid.  Let $Q\subseteq\mathbb{R}^{d+1}$ be the pyramid formed by appending $a\in\mathbb{R}\setminus\lbrace0\rbrace$ as the last coordinate to $\sqrt{1+a^2}P$, and taking the convex hull with $\left( 0,...,0,-\frac{1}{a} \right)$.  For $x=(x_1,x_2,\dots,x_d)\in\Rd$ and $y\in\mathbb{R}$, we will use the notation $(x;y)$ to stand for $(x_1,x_2,\dots,x_d,y)\in\mathbb{R}^{d+1}$.

Now we will use Theorem \ref{selfpolarconditions} to show that $Q$ is self-polar by $W$.  For the first condition, we need to identify the vertices of $Q$, which are $\left( 0,...,0,-\frac{1}{a} \right)$ and $\left(\sqrt{1+a^2}v; a\right)$ for $v\in\vertop(P)$.  Let $v,w$ be any two vertices of $P$.  Then we have
\begin{equation}\label{pyramideq1}
\begin{split}
\left\langle \left(\sqrt{1+a^2}v; a\right),W\left(\sqrt{1+a^2}w; a\right)\right\rangle
&= \left\langle \left(\sqrt{1+a^2}v; a\right),\left(\sqrt{1+a^2}Uw; -a\right)\right\rangle\\
&= \left\langle \sqrt{1+a^2}v,\sqrt{1+a^2}Uw\right\rangle-a^2\\
&= (1+a^2)\langle v,Uw\rangle-a^2\\
&\leq(1+a^2)(1)-a^2\\
&=1
\end{split}
\end{equation}

As for the apex vertex of $Q$, we have
\begin{equation}\label{pyramideq2}
\begin{split}
\left\langle \left(\sqrt{1+a^2}v; a\right),W\left( 0,...,0,-\frac{1}{a} \right)\right\rangle
&= \left\langle \left(\sqrt{1+a^2}v; a\right),\left( 0,...,0,\frac{1}{a} \right)\right\rangle\\
&= a\cdot\frac{1}{a}\\
&=1
\end{split}
\end{equation}

Hence not only do the vertices of $Q$ fulfill the first condition of Theorem \ref{selfpolarconditions}, but from Equation \ref{pyramideq2} we can see that we have also fulfilled the second condition for the facet of $Q$ that is the scaled and translated copy of $P$.

As for the other facets of $Q$, each facet $F$ of $Q$ corresponds to a facet $\hat{F}$ of $P$ in such a way that $F$ has as its vertices $\left( 0,...,0,-\frac{1}{a} \right)$ and $\left(\sqrt{1+a^2}v; a\right)$ for $v\in\vertop(\hat{F})$.  Since $P$ is self-polar by $U$, we know that for each facet $\hat{F}$ of $P$ there is a vertex $w$ of $P$ for which $\langle v,Uw\rangle=1$ for all $v\in\vertop(\hat{F})$.  Then Equation \ref{pyramideq1} applies, but with equality rather than inequality, and Equation \ref{pyramideq2} applies with $w$ substituted for $v$.  Thus we have fulfilled the second condition of Theorem \ref{selfpolarconditions} for all facets of $Q$.
\end{proof}

\subsection{Joins}

For polytopes $P\in\mathcal{P}_0^{d_1}$ and $Q\in\mathcal{P}_0^{d_2}$, the \emph{join} of $P$ and $Q$, denoted $P*Q$, is the combinatorial type of a $d_1+d_2+1$ polytope which can be realized by embedding $P$ and $Q$ into orthogonal subspaces of $\mathbb{R}^{d_1+d_2+1}$, then translating the two subspaces apart along the one-dimensional subspace of $\mathbb{R}^{d_1+d_2+1}$ that is orthogonal to both, and finally taking the convex hull.

For every face $f$ of $P$ and $g$ of $Q$, the join $f*g$ is a face of $P*Q$, and conversely, every face of $P*Q$ is the join of a face from $P$ with a face from $Q$.  Here, we are including $P$ and $Q$ as well as the empty set as faces.  (Note that $f*\emptyset=f$.)  Just as with $P$ and $Q$ themselves, the dimension of a face $f*g$ of $P*Q$ is $\dim(f)+\dim(g)+1$. \cite{handbook15}

Since a pyramid is the join of a polytope with a point, it seems natural to suppose that, as with pyramids, joins of self-polar polytopes might be self-polar.  This is indeed the case.

\begin{theorem} \label{jointhm}
For polytopes $P_1\in\mathcal{P}_0^{d_1}$ and $P_2\in\mathcal{P}_0^{d_2}$ that are self-polar by orthogonal transformations $U_1$ and $U_2$ respectively, the join of $P_1$ and $P_2$ can be realized as self-polar by the transformation
\begin{equation*}
W:=
\begin{pmatrix}
U_1 & 0 & 0 \\
0 & U_2 & 0 \\
0 & 0 & -1
\end{pmatrix}
\end{equation*}
\end{theorem}
\begin{proof}
First we construct a realization of the join.  Let $Q\subseteq\mathbb{R}^{d_1+d_2+1}$ be the join formed by the following process.  To each point of $\sqrt{1+a^2}P_1$, we append $d_2$ zeros and then $a\in\mathbb{R}\setminus\lbrace 0\rbrace$ as the last coordinate.  To each point of $\sqrt{1+\frac{1}{a^2}}P_2$, we insert $d_1$ zeros before the first coordinate, and then append $-\frac{1}{a}$ as the last coordinate.  Finally, we take the convex hull.

For $x=(x_1,x_2,\dots,x_{d_1})\in\mathbb{R}^{d_1}$, $y=(y_1,y_2,\dots,y_{d_2})\in\mathbb{R}^{d_2}$, and $z\in\mathbb{R}$, we will use the notation $(x;y;z)$ to stand for $(x_1,x_2,\dots,x_{d_1},y_1,y_2,\dots,y_{d_2},z)\in\mathbb{R}^{d_1+d_2+1}$.

Now we will use Theorem \ref{selfpolarconditions} to show that $Q$ is self-polar by $W$.  For the first condition, we need to confirm that any two vertices of $Q$ have a dot product of $\leq 1$.  Clearly the vertices of $Q$ are of two types: vertices of the embedded copy of $P_1$ and vertices of the embedded copy of $P_2$.  Let $v,w$ be any two vertices of $P_1$.  Then we have
\begin{multline}\label{joineq1}
\left\langle \left(\sqrt{1+a^2}v; 0,\dots,0; a\right),W\left(\sqrt{1+a^2}w; 0,\dots,0; a\right)\right\rangle\\
\begin{aligned}
&= \left\langle \left(\sqrt{1+a^2}v; 0,\dots,0; a\right),\left(\sqrt{1+a^2}U_1w; 0,\dots,0;-a\right)\right\rangle\\
&= \left\langle \sqrt{1+a^2}v,\sqrt{1+a^2}U_1w\right\rangle-a^2\\
&= (1+a^2)\langle v,U_1w\rangle-a^2\\
&\leq(1+a^2)(1)-a^2\\
&=1
\end{aligned}
\end{multline}

On the other hand, let $v,w$ be any two vertices of $P_2$.  Then we have
\begin{multline}\label{joineq2}
\left\langle \left(0,\dots,0; \sqrt{1+\frac{1}{a^2}}v; -\frac{1}{a}\right),W\left(0,\dots,0; \sqrt{1+\frac{1}{a^2}}w; -\frac{1}{a}\right)\right\rangle\\
\begin{aligned}
&= \left\langle \left(0,\dots,0; \sqrt{1+\frac{1}{a^2}}v; -\frac{1}{a}\right),\left(0,\dots,0; \sqrt{1+\frac{1}{a^2}}U_2w; \frac{1}{a}\right)\right\rangle\\
&= \left\langle \sqrt{1+\frac{1}{a^2}}v,\sqrt{1+\frac{1}{a^2}}U_2w\right\rangle-\frac{1}{a^2}\\
&= \left(1+\frac{1}{a^2}\right)\langle v,U_2w\rangle-\frac{1}{a^2}\\
&\leq\left(1+\frac{1}{a^2}\right)(1)-\frac{1}{a^2}\\
&=1
\end{aligned}
\end{multline}

Finally, for a vertex $v$ of $P_1$ and a vertex $w$ of $P_2$, we have
\begin{multline}\label{joineq3}
\left\langle \left(\sqrt{1+\frac{1}{a^2}}v; 0,\dots,0; a\right),W\left(0,\dots,0; \sqrt{1+\frac{1}{a^2}}w; -\frac{1}{a}\right)\right\rangle\\
\begin{aligned}
&= \left\langle \left(\sqrt{1+\frac{1}{a^2}}v; 0,\dots,0; a\right),\left(0,\dots,0; \sqrt{1+\frac{1}{a^2}}U_2w; \frac{1}{a}\right)\right\rangle\\
&=a\cdot\frac{1}{a}\\
&=1
\end{aligned}
\end{multline}

Now for the second condition of Theorem \ref{selfpolarconditions}, we need to identify the facets of $Q$.  Since the dimension of each facet must be $d_1+d_2$, and every face of $Q$ is the join of a face from $P_1$ with a face from $P_2$, each facet of $Q$ must be either the join of $P_1$ with a facet of $P_2$, or the join of a facet of $P_1$ with $P_2$.

Let $F$ be a facet of $P_1$, and let $F'$ be the corresponding facet of $Q$ which is the join of $F$ with $P_2$.  Since $P_1$ is self-polar by $U_1$, there is some vertex $w$ of $P_1$ such that $\langle v,U_1w\rangle=1$ for all $v\in\vertop(F)$.  Then Equation \ref{joineq1} applies to all such $v$, with equality rather than inequality.  As for the vertices of $F'$ which come from $P_2$ rather than from $P_1$, Equation \ref{joineq3} applies to them.

Now for the other case, let $F$ be a facet of $P_2$, and let $F'$ be the corresponding facet of $Q$ which is the join of $F$ with $P_1$.  Since $P_2$ is self-polar by $U_2$, there is some vertex $w$ of $P_2$ such that $\langle v,U_2w\rangle=1$ for all $v\in\vertop(F)$.  Then Equation \ref{joineq2} applies to all such $v$, with equality rather than inequality.  As for the vertices of $F'$ which come from $P_1$ rather than from $P_2$, Equation \ref{joineq3} applies to them.
\end{proof}

\subsection{Sections and Projections}

Now that we have shown how to construct self-polar polytopes in higher dimensions from those in lower dimensions, it is natural to consider whether we can also construct self-polar polytopes in lower dimensions from those in higher dimensions.

First, we need to establish how the polar operation works on sections and projections.

\begin{lemma}\label{lowerdimlemma}
For a polytope $P\subseteq\Rd$ and a subspace $H$ of $\Rd$, the polar of the orthogonal projection of $P$ onto $H$ is the intersection of the polar of $P$ with $H$:
\[\left(\mathrm{proj}_H(P)\right)^\circ=P^\circ\cap H\]
\end{lemma}
\begin{proof}
Note that $\mathrm{proj}_H(P)$ is given by $AA^TP$, where the columns of $A$ are an orthonormal basis of $H$.  Note also that we are taking the polar in $H$ only, not in $\Rd$.  Since our projection is orthogonal, we have
\begin{align*}
\left(\mathrm{proj}_H(P)\right)^\circ
&= \{y\in H : \langle y, AA^Tx\rangle\leq 1 \textup{ for all } x\in P\}\\
&= \{y\in H : \langle AA^Ty, x\rangle\leq 1 \textup{ for all } x\in P\}\\
&= \{y\in H : \langle y, x\rangle\leq 1 \textup{ for all } x\in P\}\\
&= P^\circ\cap H
\end{align*}\end{proof}

Now we can state conditions for the existence of a lower-dimensional self-polar projection or cross-section.

\begin{theorem}\label{lowerdimthm}
For a polytope $P\in\Pzd$ that is self-polar by orthogonal transformation $U$, and given a subspace $H$ of $\Rd$, if $UH=H$ and $\mathrm{proj}_H(P)=P\cap H$, then $P\cap H$ is self-polar by the restriction of $U$ to $H$.
\end{theorem}
\begin{proof}
We know from the lemma that $\left(\mathrm{proj}_H(P)\right)^\circ=P^\circ\cap H$, so by assumption we have $(P\cap H)^\circ = P^\circ\cap H$.  Multiplying both sides by $U$, we get  $U(P\cap H)^\circ = UP^\circ\cap UH$.  Since $P=UP^\circ$ and $UH=H$, we obtain $U(P\cap H)^\circ = P\cap H$.
\end{proof}

Another consequence of Lemma \ref{lowerdimlemma} is that it gives us a condition for the existence of a self-polar pyramid for the negative transformation.

\begin{theorem} \label{negpyrbasethm}
If a pyramid $P\in\Pzd$ has a negatively self-polar realization, then so does the base of the pyramid.
\end{theorem}

\begin{proof}
Suppose $P=-P^\circ\in\Pzd$ is a pyramid over base $Q$.  Let $a\in\Rd$ be the apex of the pyramid.  From Lemma \ref{dualautolemma}, we know that the dual automorphism of the face lattice that is realized by the negative transformation must pair $a$ with $Q$.  Furthermore, from Lemma \ref{tangentlemma}, we know that $Q$ must be orthogonal to $a$.

Let $H$ be the subspace of $\Rd$ orthogonal to $a$.  We know $H=-H$ since this is true of any subspace.  From Lemma \ref{lowerdimlemma} we know that $\left(\mathrm{proj}_H(P)\right)^\circ = P^\circ\cap H$, where the polar is taken with respect to the space $H$.  Then, multiplying both sides by $-1$, we get
\begin{align*}
-\left(\mathrm{proj}_H(P)\right)^\circ &= -(P^\circ\cap H)\\
&= -P^\circ\cap -H\\
&= P\cap H
\end{align*}

Let $S$ be the intersection of $P$ with $H$.  Since the base of the pyramid, $Q$, is orthogonal to the apex, $a$, we can say that for some constant $c>0$, $\mathrm{proj}_H(P)=cS$.  Hence $-(cS)^\circ=S$.  Replacing $c$ with $\sqrt{c}\sqrt{c}$, we get $S = -(\sqrt{c}\sqrt{c})S)^\circ$ and applying Lemma \ref{polarofmatrixlemma}, we get $S = -\frac{1}{\sqrt{c}}(\sqrt{c}S)^\circ$, which yields $\sqrt{c}S=-(\sqrt{c}S)^\circ$.  Since $\sqrt{c}S$ is of the same combinatorial type as $Q$, we have a negatively self-polar realization of the base of the pyramid.
\end{proof}

\section{Vertex Numbers of Negatively Self-Polar Polytopes}\label{sec:vertexnumbers}

The theorems in the last section can be applied right away to the question of how many vertices are possible for negatively self-polar polytopes in dimensions higher than three.  Starting in three dimensions with the constructions given in Theorem \ref{d3types} and applying Theorem \ref{pyramidthm} repeatedly, we can construct negatively self-polar polytopes in higher dimensions with (almost) any number of vertices, with the only exception being $d+2$ vertices in dimension $d$.

The question of whether there exist any $d$-dimensional negatively self-polar polytopes with $d+2$ vertices is somewhat challenging, since there is no starting construction in $d=3$, due to Theorem \ref{negpyrbasethm}.  Luckily, polytopes having this number of vertices have already received a great deal of study; see \cite{grunbaum} or \cite{ziegler}, for example.

\begin{theorem}
For $d\geq 3$, there exist negatively self-polar $d$-dimensional polytopes with $n$ vertices for all values of $n\geq d+1$ except $n=d+2$.
\end{theorem}
\begin{proof}
For number of vertices $n=d+1$ and $n\geq d+3$, the base case of $d=3$ has already been shown, and the higher-dimensional polytopes can be constructed inductively using the information in the proofs of Theorems \ref{d3types} and \ref{pyramidthm}.

As for $n=d+2$, polytopes with this number of vertices are either simplicial, or are (multiple) pyramids over simplicial polytopes (see \cite{grunbaum} or \cite{ziegler}, for example). A simplicial polytope with $n=d+2$ cannot be self-dual because it would have to be both simplicial and simple, which is only true of simplices, which have $n=d+1$.

As for (multiple) pyramids over simplicial polytopes, we have already shown in Theorem \ref{negpyrbasethm} that a negatively self-polar pyramid must have a negatively self-polar base.  Hence, for a (multiple) pyramid over a simplicial base to be negatively self-polar, the simplicial base would have to be negatively self-polar in some lower dimension.  But then the base would have to be a simplex, and so the entire pyramid would have to be a simplex.
\end{proof}

For $d>3$, it is an open question whether self-dual $d$-dimensional polytopes with $d+2$ vertices are self-polar via some other orthogonal transformation.

\section{Modifications in the Same Dimension}\label{sec:modifications}

\subsection{Intermediate Construction}
The following theorem and its proof give a method for constructing self-polar polytopes for a chosen orthogonal transformation, using a series of sets contained in the transforms of their polar sets.  The theorem also establishes when a self-polar polytope exists under certain circumstances.

As the theorem statement and its proof are rather complicated, we will first look at an example in depth.

\begin{example}\label{intermediateexample}
To keep things simple, we will work in $\mathbb{R}^2$.  Let $P$ be the triangle with vertices at $(\frac{1}{2},\frac{1}{2}),(0,-1),$ and $(-1,0)$.  Then $-P^\circ$ is a triangle with vertices at $(1,1),(1,-3),$ and $(-3,1)$.  Both are shown in the left panel of Figure \ref{intermediate1}.  Our goal is to construct a polytope $Q=-Q^\circ$ that is between $P$ and $-P^\circ$ in the sense that $P\subset Q=-Q^\circ \subset -P^\circ$.

\begin{figure}
\begin{minipage}[c]{0.5\textwidth}
\centering
\begin{tikzpicture}[x=0.85cm,y=0.85cm]
\draw[fill=gray!30] (-3,1) -- (1,1) -- (1,-3) -- cycle;
\draw[fill=gray] (1/2,1/2) -- (0,-1) -- (-1,0) -- cycle;
\draw[gray!60] (1,2) -- (1,-3);
\draw[gray!60] (2,2) -- (2,-3);
\draw[gray!60] (-1,2) -- (-1,-3);
\draw[gray!60] (-2,2) -- (-2,-3);
\draw[gray!60] (-3,2) -- (-3,-3);
\draw[gray!60] (2,1) -- (-3,1);
\draw[gray!60] (2,2) -- (-3,2);
\draw[gray!60] (2,-1) -- (-3,-1);
\draw[gray!60] (2,-2) -- (-3,-2);
\draw[gray!60] (2,-3) -- (-3,-3);
\draw[<->] (-3,0) -- (2,0) node[right] {$x$};
\draw[<->] (0,-3) -- (0,2) node[above] {$y$};
\node[below left] at (0,0) {$P$};
\node[above right] at (-2,0) {$-P^\circ$};
\draw (-3,1) -- (1,1) -- (1,-3) -- cycle;
\draw (1/2,1/2) -- (0,-1) -- (-1,0) -- cycle;
\end{tikzpicture}
\end{minipage}
\begin{minipage}[c]{0.5\textwidth}
\centering
\begin{tikzpicture}[x=0.85cm,y=0.85cm]
\draw[<->] (-3,0) -- (2,0) node[right] {$x$};
\draw[<->] (0,-3) -- (0,2) node[above] {$y$};
\node[below left] at (1,1) {$S_1$};
\node[below right] at (-1,1) {$S_2$};
\node[above right] at (-1,-1) {$S_3$};
\node[above left] at (1,-1) {$S_4$};
\end{tikzpicture}
\end{minipage}%
\caption{On the left, $P$ and $-P^\circ$; on the right, the auxiliary sets $S_1,S_2,S_3,S_4$.}\label{intermediate1}\end{figure}
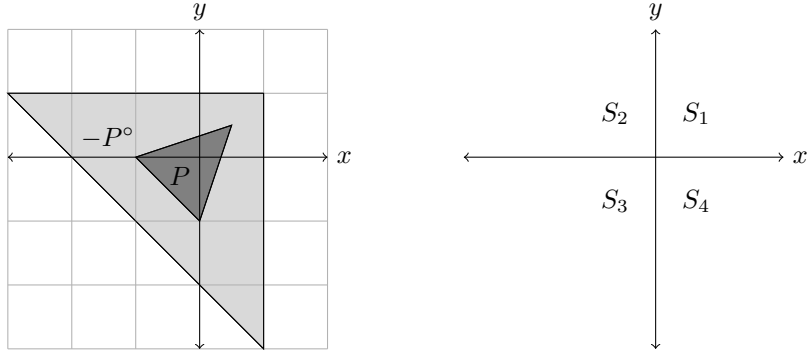

The general strategy is to enlarge $P$ while still ensuring that it remains a subset of $-P^\circ$.  To this end, we will use four auxiliary sets, $S_1,S_2,S_3,S_4$.  Each auxiliary set is one of the four quadrants, as shown in Figure the right panel of \ref{intermediate1}.  These sets are convenient because each $S_i=-S_i^\circ$.

We will define $P_0:=P$ and then successively create enlarged versions of $P$ by letting $P_1=[(P_0\cup S_1)\cap -P_0^\circ]$, $P_2=[(P_1\cup S_2)\cap -P_1^\circ]$, and so on.  The results for $P_1$ are shown in Figure \ref{intermediate2}, and the results for $P_2$ are shown in Figure \ref{intermediate3}.

As it turns out, we already have our desired goal for $Q=P_2$, so $S_3$ and $S_4$ are unnecessary.  Depending on which of the auxiliary sets we used, however, and in which order, we could obtain several different suitable sets to fill the role of $Q$.

\begin{figure}
\begin{minipage}[c]{0.33\textwidth}
\centering
\begin{tikzpicture}[x=0.75cm,y=0.75cm]
\draw[fill=gray!30] (-3,1) -- (1,1) -- (1,-3) -- cycle;
\draw[fill=gray] (1/2,1/2) -- (0,-1) -- (-1,0) -- cycle;
\draw[gray!60] (1,2) -- (1,-3);
\draw[gray!60] (2,2) -- (2,-3);
\draw[gray!60] (-1,2) -- (-1,-3);
\draw[gray!60] (-2,2) -- (-2,-3);
\draw[gray!60] (-3,2) -- (-3,-3);
\draw[gray!60] (2,1) -- (-3,1);
\draw[gray!60] (2,2) -- (-3,2);
\draw[gray!60] (2,-1) -- (-3,-1);
\draw[gray!60] (2,-2) -- (-3,-2);
\draw[gray!60] (2,-3) -- (-3,-3);
\draw[<->] (-3,0) -- (2,0) node[right] {$x$};
\draw[<->] (0,-3) -- (0,2) node[above] {$y$};
\node[below left] at (0,0) {$P_0$};
\node[above] at (0.5,-2) {$-P_0^\circ$};
\draw (-3,1) -- (1,1) -- (1,-3) -- cycle;
\draw (1/2,1/2) -- (0,-1) -- (-1,0) -- cycle;
\draw[red] (2,0) -- (0,0) -- (0,2);
\fill[pattern=north west lines, pattern color=red] (2,0) -- (0,0) -- (0,2) -- (2,2) -- cycle;
\node[above, red] at (1,2) {$S_1$};
\end{tikzpicture}
\end{minipage}
\begin{minipage}[c]{0.33\textwidth}
\centering
\begin{tikzpicture}[x=0.75cm,y=0.75cm]
\draw[fill=gray!30] (-3,1) -- (1,1) -- (1,-3) -- cycle;
\draw[fill=gray] (-1,0) -- (0,-1) -- (1/3,0) -- (1,0) -- (1,1) -- (0,1) -- (0,1/3) -- cycle;
\draw[gray!60] (1,2) -- (1,-3);
\draw[gray!60] (2,2) -- (2,-3);
\draw[gray!60] (-1,2) -- (-1,-3);
\draw[gray!60] (-2,2) -- (-2,-3);
\draw[gray!60] (-3,2) -- (-3,-3);
\draw[gray!60] (2,1) -- (-3,1);
\draw[gray!60] (2,2) -- (-3,2);
\draw[gray!60] (2,-1) -- (-3,-1);
\draw[gray!60] (2,-2) -- (-3,-2);
\draw[gray!60] (2,-3) -- (-3,-3);
\draw[<->] (-3,0) -- (2,0) node[right] {$x$};
\draw[<->] (0,-3) -- (0,2) node[above] {$y$};
\node[align=center] at (1,1) {\small $(P_0\cup S_1)$\\$\cap\,-P_0^\circ$};
\node[above] at (0.5,-2) {$-P_0^\circ$};
\draw (-3,1) -- (1,1) -- (1,-3) -- cycle;
\draw (-1,0) -- (0,-1) -- (1/3,0) -- (1,0) -- (1,1) -- (0,1) -- (0,1/3) -- cycle;
\end{tikzpicture}
\end{minipage}%
\begin{minipage}[c]{0.33\textwidth}
\centering
\begin{tikzpicture}[x=0.75cm,y=0.75cm]
\draw[fill=gray!30] (1,1) -- (-1,1) -- (-1,0) -- (0,-1) -- (1,-1) -- cycle;
\draw[fill=gray] (-1,0) -- (0,-1) -- (1,0) -- (1,1) -- (0,1) -- cycle;
\draw[gray!60] (1,2) -- (1,-3);
\draw[gray!60] (2,2) -- (2,-3);
\draw[gray!60] (-1,2) -- (-1,-3);
\draw[gray!60] (-2,2) -- (-2,-3);
\draw[gray!60] (-3,2) -- (-3,-3);
\draw[gray!60] (2,1) -- (-3,1);
\draw[gray!60] (2,2) -- (-3,2);
\draw[gray!60] (2,-1) -- (-3,-1);
\draw[gray!60] (2,-2) -- (-3,-2);
\draw[gray!60] (2,-3) -- (-3,-3);
\draw[<->] (-3,0) -- (2,0) node[right] {$x$};
\draw[<->] (0,-3) -- (0,2) node[above] {$y$};
\node[above right] at (0,0) {$P_1$};
\node[below] at (0.5,-1) {$-P_1^\circ$};
\draw (1,1) -- (-1,1) -- (-1,0) -- (0,-1) -- (1,-1) -- cycle;
\draw (-1,0) -- (0,-1) -- (1,0) -- (1,1) -- (0,1) -- cycle;
\end{tikzpicture}
\end{minipage}
\caption{The process of constructing $P_1=[(P_0\cup S_1)\cap -P_0^\circ]$ and $-P_1^\circ$.}\label{intermediate2}\end{figure}
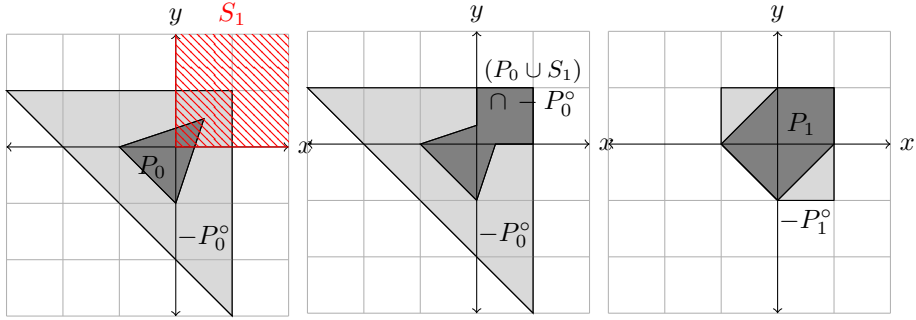

\begin{figure}
\begin{minipage}[c]{0.33\textwidth}
\centering
\begin{tikzpicture}[x=0.75cm,y=0.75cm]
\draw[fill=gray!30] (1,1) -- (-1,1) -- (-1,0) -- (0,-1) -- (1,-1) -- cycle;
\draw[fill=gray] (-1,0) -- (0,-1) -- (1,0) -- (1,1) -- (0,1) -- cycle;
\draw[gray!60] (1,2) -- (1,-3);
\draw[gray!60] (2,2) -- (2,-3);
\draw[gray!60] (-1,2) -- (-1,-3);
\draw[gray!60] (-2,2) -- (-2,-3);
\draw[gray!60] (-3,2) -- (-3,-3);
\draw[gray!60] (2,1) -- (-3,1);
\draw[gray!60] (2,2) -- (-3,2);
\draw[gray!60] (2,-1) -- (-3,-1);
\draw[gray!60] (2,-2) -- (-3,-2);
\draw[gray!60] (2,-3) -- (-3,-3);
\draw[<->] (-3,0) -- (2,0) node[right] {$x$};
\draw[<->] (0,-3) -- (0,2) node[above] {$y$};
\node[above right] at (0,0) {$P_1$};
\node[below] at (0.5,-1) {$-P_1^\circ$};
\draw (1,1) -- (-1,1) -- (-1,0) -- (0,-1) -- (1,-1) -- cycle;
\draw (-1,0) -- (0,-1) -- (1,0) -- (1,1) -- (0,1) -- cycle;
\draw[red] (-3,0) -- (0,0) -- (0,2);
\fill[pattern=north west lines, pattern color=red] (-3,0) -- (0,0) -- (0,2) -- (-3,2) -- cycle;
\node[above, red] at (-2,2) {$S_2$};
\end{tikzpicture}
\end{minipage}
\begin{minipage}[c]{0.33\textwidth}
\centering
\begin{tikzpicture}[x=0.75cm,y=0.75cm]
\draw[fill=gray!30] (1,1) -- (-1,1) -- (-1,0) -- (0,-1) -- (1,-1) -- cycle;
\draw[fill=gray] (-1,1) -- (-1,0) --  (0,-1) -- (1,0) -- (1,1) -- (0,1) -- cycle;
\draw[gray!60] (1,2) -- (1,-3);
\draw[gray!60] (2,2) -- (2,-3);
\draw[gray!60] (-1,2) -- (-1,-3);
\draw[gray!60] (-2,2) -- (-2,-3);
\draw[gray!60] (-3,2) -- (-3,-3);
\draw[gray!60] (2,1) -- (-3,1);
\draw[gray!60] (2,2) -- (-3,2);
\draw[gray!60] (2,-1) -- (-3,-1);
\draw[gray!60] (2,-2) -- (-3,-2);
\draw[gray!60] (2,-3) -- (-3,-3);
\draw[<->] (-3,0) -- (2,0) node[right] {$x$};
\draw[<->] (0,-3) -- (0,2) node[above] {$y$};
\node[above,align=center] at (0,-0.3) {\small $(P_1\cup S_2)$\\$\cap\, -P_1^\circ$};
\node[below] at (0.5,-1) {$-P_1^\circ$};
\draw (1,1) -- (-1,1) -- (-1,0) -- (0,-1) -- (1,-1) -- cycle;
\draw (-1,1) --(-1,0)-- (0,-1) -- (1,0) -- (1,1) -- (0,1) -- cycle;
\end{tikzpicture}
\end{minipage}%
\begin{minipage}[c]{0.33\textwidth}
\centering
\begin{tikzpicture}[x=0.75cm,y=0.75cm]
\draw[fill=gray] (-1,1) -- (-1,0) --  (0,-1) -- (1,0) -- (1,1) -- (0,1) -- cycle;
\draw[gray!60] (1,2) -- (1,-3);
\draw[gray!60] (2,2) -- (2,-3);
\draw[gray!60] (-1,2) -- (-1,-3);
\draw[gray!60] (-2,2) -- (-2,-3);
\draw[gray!60] (-3,2) -- (-3,-3);
\draw[gray!60] (2,1) -- (-3,1);
\draw[gray!60] (2,2) -- (-3,2);
\draw[gray!60] (2,-1) -- (-3,-1);
\draw[gray!60] (2,-2) -- (-3,-2);
\draw[gray!60] (2,-3) -- (-3,-3);
\draw[<->] (-3,0) -- (2,0) node[right] {$x$};
\draw[<->] (0,-3) -- (0,2) node[above] {$y$};
\node[above] at (0,0) {$P_2=-P_2^\circ$};
\draw (-1,1) --(-1,0)-- (0,-1) -- (1,0) -- (1,1) -- (0,1) -- cycle;
\end{tikzpicture}
\end{minipage}
\caption{The process of constructing $P_2=[(P_1\cup S_2)\cap -P_1^\circ]$ and $-P_2^\circ$.}\label{intermediate3}\end{figure}
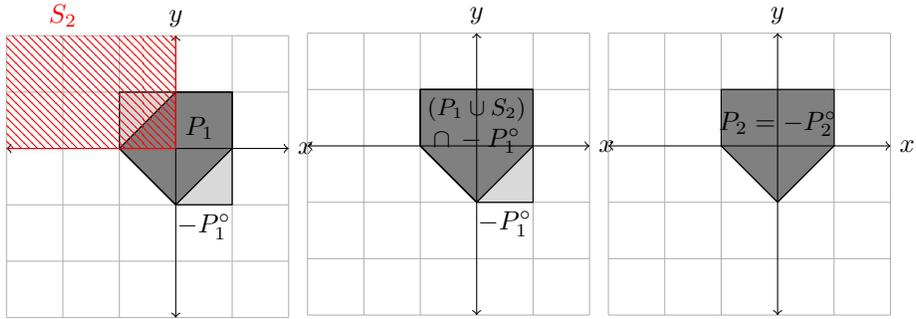

\end{example}

\begin{theorem}
\label{intermediate}
For $U$, an orthogonal transformation of $\Rd$, and a set $P\subset\Rd$, there exists a (polyhedral) set $Q$ such that $P\subseteq Q=UQ^\circ\subseteq UP^\circ$ if and only if all the following:
\begin{enumerate}
   \item There exists a closed, convex (polyhedral) set $R\subset\Rd$ such that
\begin{enumerate}
\item $U^2 R=R$
\item $P\subseteq R\subseteq UR^\circ\subseteq UP^\circ$
\end{enumerate}
   \item There exists a collection $S_1, S_2, ..., S_n$ of closed, convex (polyhedral) sets in $\Rd$ such that
\begin{enumerate}
\item $U^2 S_i=S_i$ for all $i=1, 2, ..., n$
\item $S_i\subseteq US_i^\circ$ for all $i=1, 2, ..., n$
\item $UR^\circ\subseteq S_1\cup S_2\cup\cdots\cup S_n$
\end{enumerate}
\end{enumerate}

\end{theorem}

\begin{proof}

For the first direction, assume all the conditions of the theorem.

Let $S$ be one of the $S_1, S_2, ..., S_n$ and let $T=[(R\cup S)\cap UR^\circ]$.  Note that if $S$ and $R$ are polyhedral, then $T$ is also polyhedral.  Also note that any orthogonal transformation commutes with the polar operation, with the convex hull operation, and with the closure operation.

Regarding $T$, we firstly find, using our assumptions and the commutativity of $U$, that
\begin{align*}
U^2 T &=U^2[(R\cup S)\cap UR^\circ] \\
&=[(U^2 R\cup U^2 S)\cap U^3 R^\circ] \\
&=[(R\cup S)\cap UR^\circ] \\
&=T
\end{align*}

Our next finding regarding $T$ is that $T\subseteq UT^\circ$, as shown below.  We use the theorem assumptions, the fact that $[\;]$ is a closure operator, and Lemma \ref{polarproperties} throughout, as well as basic properties of sets, and the fact that $R\subseteq UR^\circ$ is equivalent to $UR\subseteq R^\circ$ by taking the polar of both sides of the subset relation.
\begin{align*}
UT^\circ &=U[(R\cup S)\cap UR^\circ]^\circ \\
&=[(UR\cup US)\cap U^2 R^\circ]^\circ \\
&=[(UR \cup US)\cap R^\circ]^\circ \\
&=((UR \cup US)\cap R^\circ)^\circ \\
&=((UR\cap R^\circ) \cup (US\cap R^\circ))^\circ \\
&=(UR \cup (US\cap R^\circ))^\circ \\
&=UR^\circ \cap (US\cap R^\circ)^\circ \\
&=UR^\circ \cap (US^{\circ\circ}\cap R^\circ)^\circ \\
&=UR^\circ \cap (US^\circ\cup R)^{\circ\circ} \\
&=UR^\circ \cap [US^\circ\cup R] \\
&\supseteq UR^\circ \cap [S\cup R] \\
&\supseteq [UR^\circ \cap (S\cup R)] \\
&=T
\end{align*}

Regarding $T$, we lastly find that
\begin{align*}
UT^\circ\cap S &=[R\cup S]\cap UR^\circ\cap S \\
&=([R\cup S]\cap S)\cap UR^\circ \\
&=S\cap UR^\circ \\
&=S\cap UR^\circ\cap S \\
&\subseteq(R\cup S)\cap UR^\circ\cap S \\
&\subseteq T\cap S
\end{align*}
Together with $T\subseteq UT^\circ$, this implies $UT^\circ\cap S = T\cap S$.

To summarize, we have discovered $T$ such that $R\subseteq T\subseteq UT^\circ\subseteq UR^\circ$, where $T$ and $UT^\circ$ have equality on $S$, and $T$ fulfills the same assumptions that we initially made about $R$.  Hence, we proceed iteratively, using $T$ as the new $R$, and constructing a new $T$ using the same formula, but using a different $S$ to construct $T$ so as to extend the equality of $T$ and $UT^\circ$ to a new region.  Below, we describe this iterative process formally.

Let $T_0:=R$.  Then for $i\in\mathbb{N}$, let $T_i:=[(T_{i-1}\cup S_i)\cap UT_{i-1}^\circ]$.  It then follows from our previous finding about $T$ that $T_{i-1}\subseteq T_i\subseteq UT_i^\circ\subseteq UT_{i-1}^\circ$.  Hence
\begin{equation*}
R=T_0\subseteq T_1\subseteq\cdots\subseteq T_n\subseteq UT_n^\circ\subseteq\cdots\subseteq UT_1^\circ\subseteq UT_0^\circ=UR^\circ
\end{equation*}

From our third finding about $T$, we have that $T_i\cap S_i=UT_i^\circ\cap S_i$.  Then by the nesting of the sets it follows that $T_n\cap S_i=UT_n^\circ\cap S_i$ for all $i=1, 2, ..., n$.  Hence $T_n\cap (S_1\cup S_2\cup\cdots\cup S_n)=UT_n^\circ\cap (S_1\cup S_2\cup\cdots\cup S_n)$.  Since $T_n\subseteq UT^\circ_n\subseteq UR^\circ$ and by assumption this is a subset of $S_1\cup S_2\cup\cdots\cup S_n$, we have $T_n=UT^\circ_n$.

Note that so long as $R$ and the $S_1,...,S_n$ are polyhedral (i.e., finite intersections of closed half-spaces), then $T_n$ is polyhedral because it is constructed in a finite number of steps from intersections, unions, convex hulls, and closures of polyhedral sets.  This completes the proof of the first direction of the theorem.

For the second direction, assume $U$ is an orthogonal transformation of $\Rd$, $P$ is a set in $\Rd$, and assume there exists a (polyhedral) set $Q$ such that $P\subseteq Q=UQ^\circ\subseteq UP^\circ$.  We will define $R:=Q$ and $S_1:=Q$.  $S_1$ is the only $S_i$ we will need.

Since $Q=UQ^\circ$, we have $U^2 Q=Q$ from Theorem \ref{Uconditionstheorem}.  This gives us 1(a) and 2(a).  We get 1(b) and 2(b) from the assumption that $P\subseteq Q=UQ^\circ\subseteq UP^\circ$.  Finally, since $UQ^\circ=Q$, we have $UR^\circ=S_1$, which gives us 2(c).
\end{proof}

Some of the requirements of Theorem \ref{intermediate} might seem rather difficult to fulfill, and while this is true in general, when $U$ is the negative transformation, matters are simpler, as shown by the following corollary of Theorem \ref{intermediate}.

\begin{corollary}
\label{intermediatecorollary}
For a (polyhedral) set $P\subset\Rd$ such that $P\subseteq -P^\circ$, there exists a (polyhedral) set $Q$ such that $P\subseteq Q=-Q^\circ\subseteq -P^\circ$.
\end{corollary}
\begin{proof}
Here, we use Theorem \ref{intermediate} and let $R$ be $P$ itself.  Obviously $U^2=I$, so we have fulfilled conditions 1(a) and 1(b).  As for the auxiliary sets $S_i$, we can use the orthants of $\Rd$, which are polyhedral.  Each orthant $S_i=-S_i^\circ$, and together the orthants cover all of $\Rd$, so they fulfill 2(a), 2(b), and 2(c).
\end{proof}

\subsection{Add-and-Cut Constructions}

The next theorem gives us a way to add vertices to a set that is already self-polar while preserving that property.  The key idea is to add a vertex and intersect with a half-space at the same time.  See Figure \ref{addcutfigure} for an example of a successful add-and-cut operation.

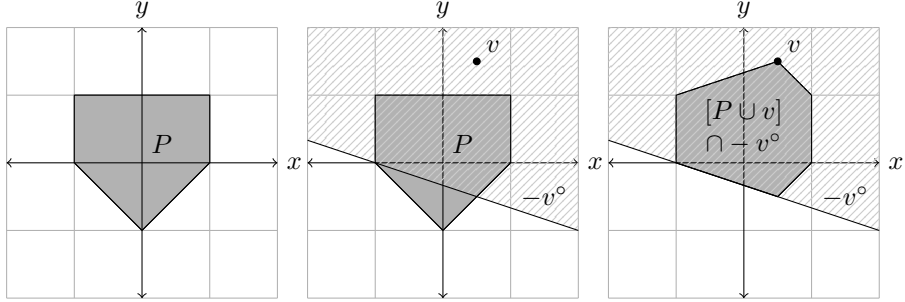
\begin{figure}

\begin{minipage}[c]{0.33\textwidth}
\centering
\begin{tikzpicture}[x=0.9cm,y=0.9cm]
\draw[fill=gray!60] (-1,1) -- (-1,0) --  (0,-1) -- (1,0) -- (1,1) -- (0,1) -- cycle;
\draw[gray!60] (1,2) -- (1,-2);
\draw[gray!60] (2,2) -- (2,-2);
\draw[gray!60] (-1,2) -- (-1,-2);
\draw[gray!60] (-2,2) -- (-2,-2);
\draw[gray!60] (2,1) -- (-2,1);
\draw[gray!60] (2,2) -- (-2,2);
\draw[gray!60] (2,-1) -- (-2,-1);
\draw[gray!60] (2,-2) -- (-2,-2);
\draw[<->] (-2,0) -- (2,0) node[right] {$x$};
\draw[<->] (0,-2) -- (0,2) node[above] {$y$};
\node[above right] at (0,0) {$P$};
\draw (-1,1) --(-1,0)-- (0,-1) -- (1,0) -- (1,1) -- (0,1) -- cycle;
\end{tikzpicture}
\end{minipage}
\begin{minipage}[c]{0.33\textwidth}
\centering
\begin{tikzpicture}[x=0.9cm,y=0.9cm]
\draw[fill=gray!60] (-1,1) -- (-1,0) --  (0,-1) -- (1,0) -- (1,1) -- (0,1) -- cycle;
\draw[gray!60] (1,2) -- (1,-2);
\draw[gray!60] (2,2) -- (2,-2);
\draw[gray!60] (-1,2) -- (-1,-2);
\draw[gray!60] (-2,2) -- (-2,-2);
\draw[gray!60] (2,1) -- (-2,1);
\draw[gray!60] (2,2) -- (-2,2);
\draw[gray!60] (2,-1) -- (-2,-1);
\draw[gray!60] (2,-2) -- (-2,-2);
\draw[<->] (-2,0) -- (2,0) node[right] {$x$};
\draw[<->] (0,-2) -- (0,2) node[above] {$y$};
\fill[pattern=north east lines, pattern color=gray!40] (-2,1/3) -- (2,-1) -- (2,2) -- (-2,2) -- cycle;
\node[above right] at (1/2,3/2) {$v$};
\draw[fill=black] (1/2,3/2) circle (0.05);
\node[above right] at (0,0) {$P$};
\node at (1.5,-0.5) {$-v^\circ$};
\draw (-1,1) --(-1,0)-- (0,-1) -- (1,0) -- (1,1) -- (0,1) -- cycle;
\draw (-2,1/3) -- (2,-1);
\end{tikzpicture}
\end{minipage}%
\begin{minipage}[c]{0.33\textwidth}
\centering
\begin{tikzpicture}[x=0.9cm,y=0.9cm]
\draw[fill=gray!60] (-1,1) -- (-1,0) --  (1/2,-1/2) -- (1,0) -- (1,1) -- (1/2,3/2) -- cycle;
\draw[gray!60] (1,2) -- (1,-2);
\draw[gray!60] (2,2) -- (2,-2);
\draw[gray!60] (-1,2) -- (-1,-2);
\draw[gray!60] (-2,2) -- (-2,-2);
\draw[gray!60] (2,1) -- (-2,1);
\draw[gray!60] (2,2) -- (-2,2);
\draw[gray!60] (2,-1) -- (-2,-1);
\draw[gray!60] (2,-2) -- (-2,-2);
\draw[<->] (-2,0) -- (2,0) node[right] {$x$};
\draw[<->] (0,-2) -- (0,2) node[above] {$y$};
\fill[pattern=north east lines, pattern color=gray!40] (-2,1/3) -- (2,-1) -- (2,2) -- (-2,2) -- cycle;
\node[above right] at (1/2,3/2) {$v$};
\draw[fill=black] (1/2,3/2) circle (0.05);
\node[above, align=center] at (0,0) {$[P\cup v]$\\$\cap-v^\circ$};
\node at (1.5,-0.5) {$-v^\circ$};
\draw (-1,1) --(-1,0)-- (1/2,-1/2) -- (1,0) -- (1,1) -- (1/2,3/2) -- cycle;
\draw (-2,1/3) -- (2,-1);
\end{tikzpicture}
\end{minipage}
\caption{Adding a vertex and cutting a new facet while preserving self-polarity.}\label{addcutfigure}\end{figure}

Caution must taken to make sure that the order in which these operations (adding a vertex and cutting a facet) are done will make no difference to the outcome; in other words, so that $[(P\cap Uv^\circ)\cup v]=[P\cup v]\cap Uv^\circ$ for the new vertex $v$.  See Figure \ref{badaddcutfigure} for an unsuccessful example where the order does make a difference.

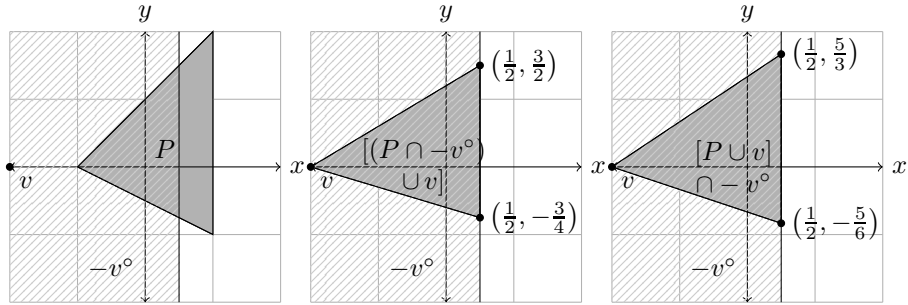
\begin{figure}
\begin{minipage}[c]{0.33\textwidth}
\centering
\begin{tikzpicture}[x=0.9cm,y=0.9cm]
\draw[fill=gray!60] (-1,0) -- (1,-1) --  (1,2) -- cycle;
\draw[gray!60] (1,2) -- (1,-2);
\draw[gray!60] (2,2) -- (2,-2);
\draw[gray!60] (-1,2) -- (-1,-2);
\draw[gray!60] (-2,2) -- (-2,-2);
\draw[gray!60] (2,1) -- (-2,1);
\draw[gray!60] (2,2) -- (-2,2);
\draw[gray!60] (2,-1) -- (-2,-1);
\draw[gray!60] (2,-2) -- (-2,-2);
\draw[<->] (-2,0) -- (2,0) node[right] {$x$};
\draw[<->] (0,-2) -- (0,2) node[above] {$y$};
\fill[pattern=north east lines, pattern color=gray!40] (1/2,2) -- (1/2,-2) -- (-2,-2) -- (-2,2) -- cycle;
\node[below right] at (-2,0) {$v$};
\draw[fill=black] (-2,0) circle (0.05);
\node[above right] at (0,0) {$P$};
\node at (-0.5,-1.5) {$-v^\circ$};
\draw (-1,0) -- (1,-1) --  (1,2) -- cycle;
\draw (1/2,2) -- (1/2,-2);
\end{tikzpicture}
\end{minipage}
\begin{minipage}[c]{0.33\textwidth}
\centering
\begin{tikzpicture}[x=0.9cm,y=0.9cm]
\draw[fill=gray!60] (-2,0) -- (1/2,-3/4) --  (1/2,3/2) -- cycle;
\draw[gray!60] (1,2) -- (1,-2);
\draw[gray!60] (2,2) -- (2,-2);
\draw[gray!60] (-1,2) -- (-1,-2);
\draw[gray!60] (-2,2) -- (-2,-2);
\draw[gray!60] (2,1) -- (-2,1);
\draw[gray!60] (2,2) -- (-2,2);
\draw[gray!60] (2,-1) -- (-2,-1);
\draw[gray!60] (2,-2) -- (-2,-2);
\draw[<->] (-2,0) -- (2,0) node[right] {$x$};
\draw[<->] (0,-2) -- (0,2) node[above] {$y$};
\fill[pattern=north east lines, pattern color=gray!40] (1/2,2) -- (1/2,-2) -- (-2,-2) -- (-2,2) -- cycle;
\node[below right] at (-2,0) {$v$};
\draw[fill=black] (-2,0) circle (0.05);
\node[right] at (1/2,-3/4) {$\left(\frac{1}{2},-\frac{3}{4}\right)$};
\draw[fill=black] (1/2,-3/4) circle (0.05);
\node[right] at (1/2,3/2) {$\left(\frac{1}{2},\frac{3}{2}\right)$};
\draw[fill=black] (1/2,3/2) circle (0.05);
\node[align=center] at (-1/3,0) {$[(P\cap -v^\circ)$\\$\cup\, v]$};
\node at (-0.5,-1.5) {$-v^\circ$};
\draw (-2,0) -- (1/2,-3/4) --  (1/2,3/2) -- cycle;
\draw (1/2,2) -- (1/2,-2);
\end{tikzpicture}
\end{minipage}%
\begin{minipage}[c]{0.33\textwidth}
\centering
\begin{tikzpicture}[x=0.9cm,y=0.9cm]
\draw[fill=gray!60] (-2,0) -- (1/2,-5/6) --  (1/2,2*5/6) -- cycle;
\draw[gray!60] (1,2) -- (1,-2);
\draw[gray!60] (2,2) -- (2,-2);
\draw[gray!60] (-1,2) -- (-1,-2);
\draw[gray!60] (-2,2) -- (-2,-2);
\draw[gray!60] (2,1) -- (-2,1);
\draw[gray!60] (2,2) -- (-2,2);
\draw[gray!60] (2,-1) -- (-2,-1);
\draw[gray!60] (2,-2) -- (-2,-2);
\draw[<->] (-2,0) -- (2,0) node[right] {$x$};
\draw[<->] (0,-2) -- (0,2) node[above] {$y$};
\fill[pattern=north east lines, pattern color=gray!40] (1/2,2) -- (1/2,-2) -- (-2,-2) -- (-2,2) -- cycle;
\node[below right] at (-2,0) {$v$};
\draw[fill=black] (-2,0) circle (0.05);
\node[right] at (1/2,-5/6) {$\left(\frac{1}{2},-\frac{5}{6}\right)$};
\draw[fill=black] (1/2,-5/6) circle (0.05);
\node[right] at (1/2,2*5/6) {$\left(\frac{1}{2},\frac{5}{3}\right)$};
\draw[fill=black] (1/2,2*5/6) circle (0.05);
\node[align=center] at (-1/5,0) {$[P\cup v]$\\$\cap -v^\circ$};
\node at (-0.5,-1.5) {$-v^\circ$};
\draw (-2,0) -- (1/2,-5/6) --  (1/2,2*5/6) -- cycle;
\draw (1/2,2) -- (1/2,-2);
\end{tikzpicture}
\end{minipage}
\caption{An unsuccessful choice of $v$ that results in $[(P\cap -v^\circ)\cup v]\neq[P\cup v]\cap -v^\circ$.}\label{badaddcutfigure}\end{figure}

\begin{theorem}\label{addcuttheorem}
For $U$, an orthogonal transformation of $\Rd$, and a set $P\subset\Rd$ such that $P=UP^\circ$, if there is a point $x\in\Rd$ such that $[(P\cap Ux^\circ)\cup x]=[P\cup x]\cap Ux^\circ$ and $U^2x=x$, then $Q=UQ^\circ$ for the set $Q=[P\cup x]\cap Ux^\circ$.
\end{theorem}
\begin{proof}
Assume that $P$ and $x$ are as described.  Then
\begin{align*}
UQ^\circ &=  \left([UP\cup Ux]\cap U^2x^\circ\right)^\circ \\
&=  \left((UP\cup Ux)^{\circ\circ}\cap x^\circ\right)^\circ \\
&=  \left((UP\cup Ux)^{\circ}\cup x\right)^{\circ\circ} \\
&=  \left((UP^\circ\cap Ux^\circ) \cup x\right)^{\circ\circ} \\
&=  \left((P\cap Ux^\circ) \cup x\right)^{\circ\circ} \\
&=  \left[(P\cap Ux^\circ) \cup x\right] \\
&= [P\cup x]\cap Ux^\circ \\
&= Q
\end{align*}
\end{proof}

To use Theorem \ref{addcuttheorem}, an easy way to find a point $x$ which can be added to a polytope is to choose a facet, take a point in the interior of the facet, and slightly increase the radius of that point to get $x$.  As long as the facet chosen does not contain the vertex to which it is dual, and as long as $x$ is still beneath the planes of all the other facets, the conditions of the theorem will be met.  A pyramid will be made over the facet, with $x$ at the apex, and the dual vertex to the facet will be cut off to form a new facet.  We state this more formally as a corollary now, which will be of much use in the investigation of combinatorial types in the next section.

\begin{corollary}
For $U$, an involutory orthogonal transformation of $\Rd$, and a polytope $P\subset\Rd$ such that $P=UP^\circ$, if vertex $v$ does not lie in its own dual facet $f$, then a polytope $Q=UQ^\circ$ exists which is equal to $P$ but with the following modifications:
\begin{enumerate}
\item A new vertex beyond $f$ but beneath all other facets creates a pyramid over $f$
\item A new facet cuts off $v$ by slicing through the interior of every face containing $v$
\end{enumerate}
\end{corollary}
\begin{proof}
Using Theorem \ref{addcuttheorem}, we let $x$ be some point which is beyond $f$ but beneath all other facets.  We have $U^2x=x$ by assumption that $U$ is an involution.  As for the other requirement, consider the following.  Because $x$ is beneath all facets except $f$, taking the convex hull of $P$ with $x$ only affects and depends on points beyond facet $f$.  On the other hand, since $x$ is beneath all facets except $f$, that means $v$ is the only vertex that is beyond $Ux^\circ$.  Hence the intersection $P\cap Ux^\circ$ only affects and depends on points strictly beneath facet $f$.  Since the two modifications are thus separated by the hyperplane through $f$, the result of performing the two modifications in either order is the same, and so $[(P\cap Ux^\circ)\cup x]=[P\cup x]\cap Ux^\circ$, fulfilling the conditions of the previous theorem.  The description of modification 1 is a consequence of choosing $x$ beyond $f$ but beneath all other facets, and modification 2 is simply the dual description of modification 1.
\end{proof}

\section{Applications}\label{sec:applications}

\subsection{Vertex Graph Coloring}

One practical application for negatively self-polar polytopes is in generating graphs with high chromatic number $\chi$ which do not have any $\chi$-cliques.  These graphs can be used as test inputs for algorithms that calculate chromatic number or as counter-examples for conjectures regarding chromatic number.

First, we need a few definitions.  We assume the reader's basic familiarity with graphs.  The \emph{chromatic number} $\chi(G)$ of a graph $G$ is the minimal number of colors necessary to assign one color to each vertex in such a way that no edge connects two vertices of the same color.  A $k$\emph{-clique} is a set of $k$ vertices such that every pair is connected.  Obviously, $\chi(G)$ must be at least as large as the largest clique in $G$.

Let $G_P$ be the graph formed from the vertices of a $d$-dimensional polytope $P=-P^\circ$, where two vertices $v$ and $w$ are connected iff $\langle v,w\rangle=-1$.  We will first present a lemma that gives a necessary condition for the existence of a clique, and then we will use the lemma to prove our main theorem for this section.

\begin{lemma}\label{cliquelemma}
If $G_P$ has a $k$-clique $\{v_1,v_2,\dots,v_k\}$, then the cross-section of $P$ given by $P\cap \aff\big(\{v_1,v_2,\dots,v_k\}\big)$ is a $(k-1)$-dimensional simplex with vertices $\{v_1,v_2,\dots,v_k\}$.
\end{lemma}
\begin{proof}
Assume $V:=\{v_1,v_2,\dots,v_k\}$ form a clique of $G_P$.  First we will show that $P\cap \aff(V)=\conv(V)$.  Suppose that $x\in\aff(V)$ but $x\notin\conv(V)$.  Then there exist $a_1,a_2,\dots,a_k\in\mathbb{R}$ such that $x=a_1v_1+\dots+a_kv_k$ and $a_1+\dots+a_k=1$ but there is at least one $a_i$ such that $a_i<0$.

Then we have $\langle x,-v_i\rangle = \langle a_1v_1+\dots+a_kv_k,-v_i\rangle
= a_1\langle v_1,-v_i\rangle+\dots+a_k\langle v_k,-v_i\rangle$.  Since $V$ is a clique, we know $\langle v_j,-v_i\rangle = 1$ unless $i=j$, in which case $\langle v_i,-v_i\rangle=-|v_i|^2$.

Thus $\langle x,-v_i\rangle=a_1+\dots+a_{i-1}-a_i|v_i|^2+a_{i+1}+\dots+a_k$.  By adding and subtracting the $a_i$ term, we can write this as $a_1+\dots+a_k-a_i-a_i|v_i|^2=1-a_i-a_i|v_i|^2$.  By assumption, $a_i<0$, so we have shown $\langle x,-v_i\rangle>1$, which means that $x\notin-P^\circ=P$.  Therefore, if a point $x$ is in $\aff(V)$ and in $P$, it must also be in $\conv(V)$, hence $P\cap\aff(V)=\conv(V)$.

Now it only remains to show that $\conv(V)$ is a $(k-1)$-dimensional simplex.  Since $|V|=k$, if $\conv(V)$ is $(k-1)$-dimensional, then it must be a simplex, because it doesn't have enough vertices to be anything else.  To show $\conv(V)$ is $(k-1)$-dimensional, we will translate $V$ so that $v_k$ is at the origin, and then show that the other $k-1$ translated vertices $\{v_1-v_k,v_2-v_k,\dots,v_{k-1}-v_k\}$ are linearly independent.

Suppose that there are some $c_1,\dots,c_{k-1}\in\mathbb{R}$ such that
\[\sum_{j=1}^{k-1}{c_j(v_j-v_k)}=0\]
Then for any $v_i\in\{v_1,\dots,v_{k-1}\}$ we must have
\begin{equation}\label{graphseceqn}
\left\langle v_i,\sum_{j=1}^{k-1}{c_j(v_j-v_k)}\right\rangle=0
\end{equation}
But we also know from the fact that $V$ is a clique that

\begin{align*}
\left\langle v_i,\sum_{j=1}^{k-1}{c_j(v_j-v_k)}\right\rangle
&=  \sum_{j=1}^{k-1}\left\langle v_i,{c_j(v_j-v_k)}\right\rangle \\
&=  \sum_{j=1}^{k-1}c_j\big(\langle v_i,v_j\rangle-\langle v_i,v_k\rangle\big)\\
\end{align*}

Since $\langle v_i,v_j\rangle-\langle v_i,v_k\rangle=0$ unless $j=i$, the summation becomes the single term $c_i\big(\langle v_i,v_i\rangle-\langle v_i,v_k\rangle\big)=c_i(|v_i|^2+1)$, which must equal $0$ due to Equation \ref{graphseceqn}.  Hence $c_i=0$, and since this is true for all $v_i\in\{v_1,\dots,v_{k-1}\}$, then $\{v_1-v_k,\dots,v_{k-1}-v_k\}$ are linearly independent, and so $\conv(V)$ is a $(k-1)$-dimensional simplex.
\end{proof}

\begin{theorem}
Let $G_P$ be the graph formed from the vertices of a $d$-dimensional polytope $P=-P^\circ$, where two vertices $v$ and $w$ are connected iff $\langle v,w\rangle=-1$.  Then the chromatic number $\chi(G_P)\geq d+1$, and $G_P$ has a $(d+1)$-clique if and only if $P$ is a simplex.
\end{theorem}
\begin{proof}
Lov\'{a}sz \cite{lovasz} showed that for a $d$-dimensional negatively self-polar polytope $P$ with vertices equidistant from the origin, $G_P$ has chromatic number $d+1$.  In his proof of that theorem, he used Lemma 4 and Theorem A of \cite{lovasz}, which showed that $\chi(G_P)\geq d+1$.  However, Lemma 4 and Theorem A did not rely on having all vertices equidistant from the origin, hence we have $\chi(G_P)\geq d+1$ more generally for all negatively self-polar polytopes.

The fact that $G_P$ has a $(d+1)$-clique if and only if $P$ is a simplex follows from Lemma \ref{cliquelemma}.
\end{proof}

\subsection{Algebra of Indicator Functions}
\label{algebrasection}
We now turn our attention to the topic of indicator functions of polytopes, and algebraic techniques which can be applied to them.  We will restrict our discussion solely to a certain class of polytopes; however, the topics below can be broadened to include all closed convex sets, and we refer the reader to \cite{barvinok} for further information.  A few definitions are necessary to begin with.

First, we will denote the \textit{indicator function} of a set $A\subseteq\Rd$ by $[A]$.  Note that we are re-using a notation which was defined differently in previous sections.  In this section, $[A]$ will always be defined as a function from $\Rd$ to $\{0,1\}$ given by
\[
[A](x)=
 \begin{cases} 
      0, & x\in A \\
      1, & x\notin A 
   \end{cases}
\]

Note that indicator functions are subject to the Principle of Inclusion and Exclusion (see, for example, \cite{barvinok}).

Next we remind the reader that $\Pzd$ is the set of all $d$-dimensional polytopes which contain the origin in their interior.  Note that this set is closed under the polar and the intersection operations, and that $P^{\circ\circ}=P$ for all $P\in\mathcal{P}$.

Finally, we define $\mathcal{A}_d$ to be the $\mathbb{R}$-algebra spanned by the indicator functions of polytopes in $\Pzd$.  Hence, an element $f\in\mathcal{A}_d$ can be written
\[f=\sum_{i=1}^m a_i[P_i]\]
where $P_i\in\Pzd$, $a_i\in\mathbb{R}$, and $m\in\mathbb{N}$.

$\mathcal{A}_d$ is an algebra because it is closed under multiplication.  This is because for $A,B\in\Pzd$ we have $[A][B]=[A\cap B]$, where $A\cap B\in\Pzd$.  Note, however, that $\mathcal{A}_d$ does not have a multiplicative identity.

Now we are ready to state our first result.

\begin{theorem}
\label{generated}
$\mathcal{A}_d$ is generated as an $\mathbb{R}$-algebra by the set of indicator functions of all $P\in\Pzd$ such that $P=-P^\circ$.
\end{theorem}
\begin{proof}
The theorem is due to two facts, which we will prove after we have discussed their consequences.  First, that every $P\in\Pzd$ is the union of a finite set of polytopes $\{Q_i\}_{i\in I}\subset\Pzd$ such that each $Q_i\subseteq -Q_i^\circ$.  Second, that every $P\in\Pzd$ such that $P\subseteq -P^\circ$ is the intersection of a finite set of polytopes $\{R_j\}_{j\in J}\subset\Pzd$ such that each $R_j=-R_j^\circ$.

The first fact means that for every $P\in\Pzd$, $[P]$ can be written, using the inclusion-exclusion principle, as
\begin{align*}
[P]
&=\left[\bigcup_{i\in I}Q_i\right]\\
&=\sum_{K\subseteq I}(-1)^{|K|+1}\left[\bigcap_{k\in K}Q_k\right]\\
&=\sum_{K\subseteq I}(-1)^{|K|+1}\prod_{k\in K}[Q_k]\\
\end{align*}

Then, using the second fact, each $Q_k$ has some finite set of polytopes $\{R_{k,j}\}_{j\in J_k}$ which are members of $\Pzd$, and whose intersection is $Q_k$, and are such that each $R_{k,j}=-R_{k,j}^\circ$.  Hence
\begin{align*}
[P]
&=\sum_{K\subseteq I}(-1)^{|K|+1}\prod_{k\in K}\left[\bigcap_{j\in J_k}R_{k,j}\right]\\
&=\sum_{K\subseteq I}(-1)^{|K|+1}\prod_{k\in K}\prod_{j\in J_k}[R_{k,j}]\\
\end{align*}
which means that the indicator function of any $P\in\Pzd$ can be written as a linear combination over $\{-1,1\}$ of products of indicator functions of negatively self-polar polytopes.  Since the indicator functions of $\Pzd$ span $\mathcal{A}_d$, this implies the theorem statement.

Now we will prove the two facts in question.  First, that every $P\in\Pzd$ is the union of a finite set of polytopes $\{Q_i\}_{i\in I}\subset\Pzd$ such that each $Q_i\subseteq -Q_i^\circ$.

Let $P\in\Pzd$ and intersect $P$ with each orthant of $\Rd$ to get $P_1,P_2,\dots,P_{2^d}$.  Clearly, $P=P_1\cup P_2\cup\cdots\cup P_{2^d}$, and we already have that each $P_i\subset -P_i^\circ$ because every orthant is negatively self-polar, but these $P_i$ are not in $\Pzd$ because they contain the origin as a vertex, not in their interior.

We can remedy this situation easily by adding a point slightly beyond the origin to each $P_i$.  We have to be careful about this, however, by choosing a point in the right direction and close enough to the origin that we still have $P_i\subseteq-P_i^\circ$ and $P_i\subset P$.

For each $P_i$, let $v_i$ be a unit vector in the interior of the orthant.  Since $P$ has the origin in its interior, there is some $\delta>0$ such that the ball of radius $\delta$ about the origin is contained in the interior of $P$.  Choose some $\mu_i>0$ such that $\mu_i\leq\delta$ and such that
\[\frac{1}{\mu_i}\geq\max_{x\in P_i}\langle x,v_i\rangle\]

Then define $p_i=-\mu_i v_i$ and $Q_i=\conv\big(P_i\cup\{p_i\}\big)$.  The fact that $|p_i|\leq\delta$ means that $Q_i\subset P$, so we still have $P=Q_1\cup Q_2\cup\cdots\cup Q_{2^d}$.  Furthermore, since $p_i$ is in the orthant negative to the orthant of $P_i$, this means that the origin is in the interior of $Q_i$, so we have $Q_i\in\Pzd$.

Now it only remains to show that we still have $Q_i\subseteq -Q_i^\circ$.   Recall that to show this, we simply need to show for any $w,x\in Q_i$ that $\langle w,x\rangle\geq-1$.  So let $w,x\in Q_i$.  Then there are some $y,z\in P_i$ and $t,\tau\in\mathbb{R}$ such that $0\leq t\leq 1$ and $0\leq\tau\leq 1$ and $w=ty+(1-t)p_i$ and $x=\tau z+(1-\tau)p_i$.  Then
\begin{align*}
\langle w,x\rangle
&=\big\langle ty+(1-t)p_i,\tau z+(1-\tau)p_i\big\rangle\\
&=t\tau\langle y,z\rangle+t(1-\tau)\langle y,p_i\rangle+(1-t)\tau\langle p_i,z\rangle+(1-t)(1-\tau)\langle p_i,p_i\rangle\\
&=t\tau\langle y,z\rangle-\mu_i t(1-\tau)\langle y,v_i\rangle-\mu_i (1-t)\tau\langle v_i,z\rangle+\mu_i^2(1-t)(1-\tau)\langle v_i,v_i\rangle\\
\end{align*}

Since $y,z$ are in the same orthant, $\langle y,z\rangle\geq 0$ and since $v_i$ is a unit vector, $\langle v_i,v_i\rangle=1$, hence
\begin{align*}
\langle w,x\rangle
&\geq -\mu_i t(1-\tau)\langle y,v_i\rangle-\mu_i (1-t)\tau\langle v_i,z\rangle+\mu_i^2(1-t)(1-\tau)\\
&\geq -\mu_i t(1-\tau)\max_{x\in P_i}\langle x,v_i\rangle-\mu_i (1-t)\tau\max_{x\in P_i}\langle x,v_i\rangle+\mu_i^2(1-t)(1-\tau)\\
&\geq -\mu_i t(1-\tau)\frac{1}{\mu_i}-\mu_i (1-t)\tau\frac{1}{\mu_i}+\mu_i^2(1-t)(1-\tau)\\
&=-t(1-\tau)-(1-t)\tau+\mu_i^2(1-t)(1-\tau)\\
&\geq -t(1-\tau)-(1-t)\tau\\
&=2t\tau-t-\tau\\
&\geq -1
\end{align*}
where the last bound can be found with some elementary calculus.

Now we have shown that every $P\in\Pzd$ is the union of a finite set of polytopes $\{Q_i\}_{i\in I}\subset\Pzd$ such that each $Q_i\subseteq -Q_i^\circ$.  To finish the proof, we need to show that every $P\subseteq -P^\circ$ is the intersection of a finite set of polytopes $\{R_j\}_{j\in J}\subset\Pzd$ such that each $R_j=-R_j^\circ$.

Let $P\in\Pzd$ and $P\subset -P^\circ$.  Then let $\{Q_1,Q_2,\dots,Q_{2^d}\}$ be the intersections of $-P^\circ$ with each orthant of $\Rd$, and let $P_i=\conv(P\cup Q_i)$ for $i=1,2,\dots,2^d$.

Let the orthant that $Q_i$ is in be denoted $\mathcal{O}_i$.
\begin{align*}
-P_i^\circ &= -(P\cup(-P^\circ\cap\mathcal{O}_i))^\circ\\
&=-P^\circ\cap-(-P^\circ\cap\mathcal{O}_i)^\circ\\
&=-P^\circ\cap-(-P^\circ\cap-\mathcal{O}_i^\circ)^\circ\\
&=-P^\circ\cap(P^\circ\cap\mathcal{O}_i^\circ)^\circ\\
&=-P^\circ\cap(P\cup\mathcal{O}_i)^{\circ\circ}\\
&\supset-P^\circ\cap(P\cup Q_i)^{\circ\circ}\\
&=-P^\circ\cap P_i\\
&=P_i
\end{align*}

Thus $P_i\subseteq-P_i^\circ$.  Then by Theorem \ref{intermediate}, there is some $R_i=-R_i^\circ$ such that $P_i\subseteq R_i=-R_i^\circ\subseteq-P_i^\circ$.

Now consider that
\begin{align*}
-P^\circ &= Q_1\cup\cdots\cup Q_{2^d}\\
&\subseteq P_1\cup\cdots\cup P_{2^d}\\
&\subseteq R_1\cup\cdots\cup R_{2^d}
\end{align*}

By taking the negative polar of these, which reverses the subset inclusions, we get
\begin{align*}
P &\supseteq -(R_1\cup\cdots\cup R_{2^d})^\circ\\
&=-R_1^\circ\cap\cdots\cap-R_{2^d}^\circ\\
&=R_1\cap\cdots\cap R_{2^d}
\end{align*}

So $R_1\cap\cdots\cap R_{2^d}\subseteq P$, and since $P\subseteq P_i\subseteq R_i$ for each $i$, we also have $P\subseteq R_1\cap\cdots\cap R_{2^d}$.  Therefore $P=R_1\cap\cdots\cap R_{2^d}$, and this completes the proof.
\end{proof}

A natural follow-up to the previous theorem is to ask what sorts of sets related to self-polar polytopes might span $\mathcal{A}_d$.  Here is one such result.

\begin{corollary}
$\mathcal{A}_d$ is spanned by the set of indicator functions of all $P\in\Pzd$ such that $P\subseteq-P^\circ$.
\end{corollary}
\begin{proof}
From Theorem \ref{generated}, we know that all the elements of $\mathcal{A}_d$ are linear combinations over $\mathbb{R}$ of products of indicator functions of negatively self-polar polytopes in $\Pzd$.  These products are the indicator functions of the intersections of the polytopes, and the intersection of a finite set of negatively self-polar polytopes is, of course, another polytope in $\Pzd$ which is contained in its negative polar, the convex hull of the union of the polytopes being intersected.
\end{proof}

We now state a corollary of a result due to \cite{L88}: there exists a linear transformation $\mathcal{D}:\mathcal{A}_d\to\mathcal{A}_d$ such that $\mathcal{D}([P])=[P^\circ]$ for any $P\in\Pzd$.  (The original result was about a larger algebra, of which $\mathcal{A}_d$ is a sub-algebra.)

We now turn to a corollary of a result presented by \cite{barvinok}.  (Again, the original result was about a larger algebra.)  Let $T:\Rd\to\Rd$ be a linear transformation.  Then there exists a linear transformation $\mathcal{T}:\mathcal{A}_d\to\mathcal{A}_d$ such that $\mathcal{T}([P])=[T(P)]$ for all $P\in\Pzd$.

Taken together, these two corollaries mean that for any unitary transformation $U$ of $\Rd$, there is a linear transformation $\mathcal{D}_U$ of $\mathcal{A}_d$ such that $\mathcal{D}_U([P])=[UP^\circ]$ for all $P\in\Pzd$.  In particular, $\mathcal{D}_U=\mathcal{U}\circ\mathcal{D}$, where $\mathcal{U}$ is the transformation of $\mathcal{A}_d$ that corresponds to $U$.

For any unitary transformation $U$ of $\Rd$, there is some subset $\mathcal{B}_d\subseteq\mathcal{A}_d$ spanned by those elements $f\in \mathcal{A}_d$ for which $f=\mathcal{D}_U(f)$.  There is also a subset $\mathcal{C}_d\subseteq\mathcal{A}_d$ spanned by the indicator functions of polytopes $P\in\Pzd$ such that $P=UP^\circ$.  Evidently, $\mathcal{C}_d\subseteq\mathcal{B}_d$, and it is natural to wonder whether in fact $\mathcal{C}_d=\mathcal{B}_d$.  For now, however, this remains an open question.

\section{Conclusions and Further Questions}
\label{sec:final}
In general, as with all sets of polytopes, we would like to know something about which combinatorial types are possible.  We have already answered the questions of possible vertex numbers of negatively self-polar polytopes, and we have established the possible $f$-vectors of negatively self-polar polytopes in two and three dimensions.  But the following question remains.
\begin{question}
What $f$-vectors are possible for negatively self-polar polytopes in $\mathbb{R}^4$ and higher dimensions?
\end{question}
More generally, we would like to know about the place of self-polar polytopes within the broader class of self-dual polytopes.  We established that all two and three dimensional self-dual polytopes have self-polar realizations, but we have no analogous result for higher dimensions.
\begin{question}
Does every self-dual polytope have a self-polar realization?
\end{question}
Should the answer be yes in all dimensions, then the next question would also be of interest.  By definition, every self-dual polytope has a dual automorphism on its face lattice.  For negatively self-polar polytopes, this map is an involution.  In \cite{grunshep}, Gr\"{u}nbaum and Shephard defined the rank of a self-dual polytope as the minimum period of all such maps.
\begin{question}
Are all self-dual polytopes with a rank $r$ duality map realizable as self-polar with an $r$-periodic orthogonal map?
\end{question}
In two dimensions, the answer is again clearly yes, since all polygons are self-dual with rank 2, and are realizable as self-polar by reflection over a single axis.

Lov\'{a}sz \cite{lovasz} showed that for a negatively self-polar polytope in $\Rd$ with vertices equidistant from the origin, the main diagonals (diagonals from a vertex to the vertices on its dual facet) all have the same length.  He further showed that the graph formed by the main diagonals has chromatic number $d+1$.
\begin{question}
What is the chromatic number of the graph formed by the main diagonals of a negatively self-polar polytope?
\end{question}
Lov\'{a}sz's study of these polytopes was motivated by questions about the chromatic number of $G(d,\alpha)$, the graph on the points of $S^{d-1}$ formed by connecting two points iff their distance is exactly $\alpha$, which is a subgraph of Borsuk's graph.  Lov\'{a}sz showed that if there exists a negatively self-polar polytope whose vertices are all $\sqrt{2/(\alpha^2-2)}$ from the origin, then the chromatic number of $G(d,\alpha)$ is $d+1$.  He left the following as an open question, however.
\begin{question}
For which values of $r$ and $d$ does a negatively self-polar polytope exist in $\Rd$ with all vertices $r$-distant from the origin?
\end{question}


\end{document}